\documentclass[preprint,onefignum,onetabnum]{siamart220329}

\usepackage{geometry,graphicx,amssymb,amsmath,amsbsy,eucal,amsfonts,amscd,bm,tcolorbox}
\usepackage{mathrsfs}
\DeclareMathAlphabet{\mathcal}{OMS}{cmsy}{m}{n}
\usepackage[all]{xy}
\usepackage{algorithm}
\usepackage{algpseudocode} 
\usepackage{booktabs}
\usepackage{pgfplots}
\usepackage{tikz}

\definecolor{g1}{HTML}{2da02d}
\definecolor{r1}{HTML}{d62728}
\definecolor{b1}{HTML}{2178b5}

\definecolor{Cerulean}{HTML}{00A2E3}
\definecolor{JungleGreen}{HTML}{00A99A}
\definecolor{Orange}{HTML}{F58137}

\usepackage[capitalize,nameinlink]{cleveref}
\crefname{section}{Section}{Sections}
\crefname{subsection}{Subsection}{Sections}
\usepackage{multirow,booktabs}
\usepackage{makecell}
\usepackage{subfigure}
\usepackage{enumitem}

\usetikzlibrary{arrows,shapes}
\usetikzlibrary{arrows, decorations.markings,fit}
\usetikzlibrary{calc,3d}
\usepackage{tikz-3dplot}

\usepackage{float}
\usepackage{subcaption}

\usepackage{multirow}
\usepackage{array}

\newtheorem{assumption}[theorem]{Assumption}
\newtheorem{remark}[theorem]{Remark}
\definecolor{colorb}{HTML}{0072BD}
\definecolor{coloro}{HTML}{D95319}
\definecolor{colorg}{HTML}{77AC30}
\definecolor{colorr}{HTML}{922B21}
\definecolor{colorgr}{HTML}{909497}
\definecolor{colord}{HTML}{17202A}
\definecolor{colordgr}{HTML}{424949}

\definecolor{r1}{HTML}{d62728}
\definecolor{g1}{HTML}{2da02d}
\definecolor{b1}{HTML}{2178b5}

\tikzset{
    >=stealth',
    punkt/.style={
           rectangle,
           rounded corners,
           draw=black, very thick,
           text width=6.5em,
           minimum height=2em,
           text centered},
    pil/.style={
           ->,
           thick,
           shorten <=2pt,
           shorten >=2pt,}
}

\title{Monte Carlo quasi-interpolation of spherical data\thanks{The first author was supported in part by National Natural Sicence Foundation of China (No. 12571407, 12101310), the Fundamental Research Funds for the Central Universities (No. 30923010912), and a Jiangsu Shuangchuang Team program (No. JSSCTD202449).}}

\author{
    Zhengjie Sun\thanks{School of Mathematics and Statistics,  Nanjing University of Science and Technology, Nanjing, China (\email{zhengjiesun@njust.edu.cn}).}
    \and
    Mengyuan Lv\thanks{School of Mathematics and Statistics,  Nanjing University of Science and Technology, Nanjing, China (\email{17684501420@163.com}).}
    \and
    Xingping Sun\thanks{Department of Mathematics, Missouri State University, Springfield, MO 65897, USA (\email{xsun@missouristate.edu}).}
}

\headers{Monte Carlo quasi-interpolation}{Z. Sun, M. Lv and X. Sun}

\definecolor{g1}{HTML}{2da02d}
\definecolor{r1}{HTML}{d62728}
\definecolor{b1}{HTML}{2178b5}

\definecolor{Cerulean}{HTML}{00A2E3}
\definecolor{JungleGreen}{HTML}{00A99A}
\definecolor{Orange}{HTML}{F58137}

\numberwithin{equation}{section}

\allowdisplaybreaks[3]

\newcommand{\bbE}{\mathbb{E}}
\newcommand{\bbN}{\mathbb{N}}
\newcommand{\bbP}{\mathbb{P}}
\newcommand{\bbR}{\mathbb{R}}
\newcommand{\bbS}{\mathbb{S}}
\newcommand{\cE}{\mathcal{E}}
\newcommand{\cO}{\mathcal{O}}

\newcommand{\cQ}{\mathcal{Q}}
\newcommand{\cM}{\mathcal{M}}
\newcommand{\cF}{\mathcal{F}}

\newcommand{\calR}{\mathcal{R}}

\newcommand{\vare}{\bm{\varepsilon}}

\newcommand{\bld}[1]{\boldsymbol{#1}}

\newcommand{\bY}{\bld{Y}}

\newcommand{\bZ}{\bld{Z}}

\newcommand{\bX}{{\bm X}}

\def\qed{~\relax\ifmmode\hskip2em \Box
 \else\unskip\nobreak\hskip1em \hfill$\Box$
 \fi \newline}

\newmuskip\pFqmuskip

\newcommand*\pFq[6][8]{%
  \begingroup 
  \pFqmuskip=#1mu\relax
  \mathchardef\normalcomma=\mathcode`,
  \mathcode`\,=\string"8000
  \begingroup\lccode`\~=`\,
  \lowercase{\endgroup\let~}\pFqcomma
  {}_{#2}F_{#3}{\left(\genfrac..{0pt}{}{#4}{#5};#6\right)}%
  \endgroup
}
\newcommand{\pFqcomma}{{\normalcomma}\mskip\pFqmuskip}

\newcommand{\hgeom}[2]{\,{}_{#1\!}F_{#2}}

\def\d{\mathrm{d}}
\def\dmu{\d\mu(x)}

\def\sphHarm{\mathcal{Y}_{\ell,k}}

\def\Pro{\Phi_{\rho}}

\def\IntSph{\int_{\bbS^d}}

\def\scaleKer{\varphi_{\rho}}

\def\sphHarm{\mathcal{Y}_{\ell k}}

\def\QiE{\cQ_{\scaleKer}^{q,\bm{\varepsilon}}}
\def\Hsig{H^{\sigma}}
\def\Htau{H^{\tau}}
\def\Hatf{\widehat{f}_{\ell,k}}
\def\HatKer{\widehat{\scaleKer}(\ell)}
\def\HatQi{\widehat{\cQ_{\scaleKer}^q f}_{\ell,k}}

\def\fHsig{\|f\|_{\Hsig}}

\usepackage{geometry}
\begin{document}
\maketitle
\begin{abstract}
We establish a deterministic and stochastic spherical quasi-interpolation framework featuring scaled zonal kernels derived from radial basis functions on the ambient Euclidean space. The method incorporates both  quasi-Monte Carlo and Monte Carlo quadrature rules to construct easily computable quasi-interpolants, 
which provide efficient approximation to Sobolev-space functions for both clean and noisy data. To enhance the approximation power and robustness of our quasi-interpolants, we develop a multilevel method in which quasi-interpolants constructed with graded resolutions join force to reduce the error of approximation. In addition, we derive probabilistic concentration inequalities for our quasi-interpolants in pertinent stochastic settings. The construction of our quasi-interpolants does not require solving 
any linear system of equations. Numerical experiments show that our quasi-interpolation algorithm is more stable and robust against noise than comparable ones in the literature. 
\end{abstract}

\section{Introduction}

Spherical geometries naturally arise from a wide range of scientific fields, including geophysics, astronomy, computer graphics, and data science. Developing efficient and accurate function approximation methods is fundamental for solving problems stemming from these curved domains. Over the last few years, various approaches have been developed
for function approximation on spherical domains including interpolation, quasi-interpolation, and hyperinterpolation. Interpolation usually requires solving systems of linear equations, which are prone to become computationally expensive for large datasets \cite{an_SINUM2012_regularized,golitschek-2001ConApprox-interpolation,hesse-2017NM-radial,jetter-1999MCoM-error,narcowich-2007FoCM-direct,narcowich-2002SIMA-scattered}. Methods encompassing hyperinterpolation and filtered hyperinterpolation entail using quadrature rules of high orders
\cite{an_SINUM2012_regularized,lin_SINUM2021_distributed,montufar_2022FoCom_distributed,sloan_1995JAT_polynomial,sloan_2012Geom_filtered,sloan2011polynomial}. Therefore, they inevitably encounter instabilities in computing 
high-order spherical harmonics and difficulties in constructing quadrature rules for high-degree polynomials. 

In contrast, quasi-interpolation methods offer computational efficiency by avoiding the need to solve linear systems \cite{buhmann2003radial-book,buhmann2022quasi-book,gao-2024NM-quasi,ramming-2018MCoM-kernel,sun-2022JSC-convergent,wu2005generalized,wu1994shape}.
Several spherical quasi-interpolation approaches have been developed, as summarized in the recent book by Buhamann and J\"{a}ger \cite{buhmann2022quasi-book}. These include tensor product trigonometric splines in spherical coordinates \cite{ganesh2006quadrature}, Fourier coefficient truncation \cite{gomes2001approximation}, and spline methods based on sphere triangulations \cite{ibanez2010construction}. 

Recently, authors of \cite{sun-gao-sun} introduced a new spherical quasi-interpolation method using scaled zonal kernels, which is constructed in two steps. In the first step, a spherical convolution operator featuring a scaled zonal kernel is employed to approximate a target function with optimal accuracy. In the second step, the underlying convolution integral is discretized using  a carefully-designed positive quadrature rule
to produce the final quasi-interpolant. We will use the phrase ``scaled kernel quasi-interpolation" (SKQI) to  refer to this particular quasi-interpolation method. 
Numerical experiments showed that SKQI outperforms methods associated with hyperinterpolation in terms of robustness against noisy data and computational efficiency \cite{sun-gao-sun}.
 
 Due to the use of high-order quadrature rules, the approach taken 
in \cite{sun-gao-sun}  still suffers
 from the computational instability and implementational inefficiency aforementioned for hyperinterpolation and filtered hyperinterpolation.
In addition, the authors of \cite{sun-gao-sun} only addressed clean data, which has impeded far-reaching applications of the SKQI method. 
Furthermore, there are other desirable features of SKQI which have not been analyzed. For example,  
 scaled zonal kernels are versatile for designing multilevel approximation algorithms, which are frequently utilized to reduce computational complexity and enhance approximation power by
 combining quasi-interpolants of graded resolutions; see e.g.\cite{farrell-2017IMAJNA-multilevel,floater-1996JCAM-multistep,franz-wendland2023multilevel,georgoulis-2013SISC-multilevel,gia_2010SINUM_multiscale,iske-2005NumerAlgor-multilevel,kempf-2023NM-high,narcowich-1999ACHA-multilevel,usta-levesley2018multilevel,wendland-2010NM-multiscale}. This makes the SKQI method effective for problems involving large data.

 Multilevel quasi-interpolation has been extensively studied in Euclidean domains \cite{franz-wendland2023multilevel,usta-levesley2018multilevel,sharon-2023SISC-multiscale}. However, these theories and techniques do not apply directly in 
 spherical domains. To a large extent, multilevel quasi-interpolation in spherical domains, especially 
  those based on linear combinations of function values and kernel evaluations, remains heretofore unexplored.

We hope to accomplish the following two tasks in
writing this paper. 1. Fortify the theoretical framework of SKQI by adopting Monte Carlo (MC) and quasi-Monte Carlo (QMC) methods, which includes multilevel algorithms where they are pertinent. 2. Enhance the practicalities of SKQI applications by analyzing both clean and noisy data.
On the deterministic front, we employ quasi-Monte Carlo quadrature rules developed in \cite{brauchart-2014MCoM-qmc} to discretize the convolution integral to produce quasi-interpolants. Under the conditions specified in Assumptions \ref{assump:kernel} and \ref{assump2} on the underlying kernel, we establish Sobolev error estimates for the SKQI method.
Furthermore, we develop a multilevel quasi-interpolation method using scaled zonal kernels and provide a comprehensive theoretical analysis of the multilevel scheme, including the derivation of recursive relations for both approximation and error operators, as well as detailed convergence properties of the hierarchical construction.
On the stochastic front, we discretize the spherical convolution integral in the Monte Carlo way. The bounded difference inequality from probability concentration theory allows us to establish an $L_2$-probability concentration inequality for the Monte Carlo quasi interpolants (MCQI). Based on this result, we derive both 
$L_2$- and $L_\infty$-convergence rates of MCQI. To motivate far-reaching applications of the MCQI method, we provide a thorough $L_2$-probabilistic convergence analysis, considering noisy data in both deterministic and stochastic settings.  Comprehensive numerical experiments demonstrate that the proposed method exhibits significant advantages over established techniques such as hyperinterpolation and filtered hyperinterpolation, especially in challenging computing environments where noisy data prevail. Overall, we believe the MCQI method is a robust and practical solution for a variety of spherical approximation problems.

The organization of the paper is as follows. In the next section, we introduce background information and notations.
In \cref{sec:sph_qi}, we study spherical quasi-interpolation methods and derive error estimates in the deterministic setting. We also introduce the multilevel quasi-interpolation on the sphere. \cref{sec:Sto_qi_sph} focuses on stochastic quasi-interpolation using Monte Carlo sampling.  \cref{sec:qi_noisydata} analyzes the 
capacity of the SKQI method in terms of handling noisy data and derive $L_2$ probabilistic convergence rates for both deterministic and stochastic sampling. In \cref{sec:numer_exper}, we present numerical experiments that validate our theoretical results, which includes: (i) convergence tests on various point sets; (ii) a comparison with filtered hyperinterpolation for approximating noisy data.

\section{Preliminaries}\label{sec:prelim}In this section, we introduce some necessary notations and definitions for function spaces on the sphere, as well as the scaled zonal kernels.
\subsection{Function spaces on the sphere}
The unit sphere is defined by $\bbS^d=\{x\in\bbR^{d+1}:\|x\|=1\}$, where $\|\cdot\|$ denotes the Euclidean norm. 
Let $L_2(\bbS^d)$ be the space of square-integrable functions on $\bbS^d$ endowed with the inner product 
\begin{equation*}
    (f,g)_{L_2(\bbS^d)}:=\IntSph f(x)g(x)\dmu,
\end{equation*}
where $\dmu$ denotes the rotationally invariant measure on $\bbS^d$. The space $L_2(\bbS^d)$ admits an orthonormal basis of spherical harmonics
\begin{equation*}
    \big\{\sphHarm:k=1,2,\ldots,Z(d,\ell),~\ell=0,1,\ldots\big\}
\end{equation*}
with dimension $Z(d,\ell)$ given by
$$Z(d,0)=1,~\text{and}~~Z(d,\ell)=\frac{(2\ell+d-1)\Gamma(\ell+d-1)}{\Gamma(\ell+1)\Gamma(d)}~~\text{for} ~\ell\geq 1.$$
The asymptotic relation $Z(d,\ell)=\cO(\ell^{d-1})$ holds true for $\ell \rightarrow\infty$. A spherical harmonic of degree $\ell$ is the restriction to $\bbS^d$ of a harmonic homogeneous polynomial of degree $\ell$ on $\bbR^{d+1}$. 
Every function in $L_2(\bbS^d)$ has a Fourier-Legendre series representation in spherical harmonics, 
\begin{equation}\label{eq:Fourier_representation}
	f(x)=\sum_{\ell=0}^{\infty}\sum_{k=1}^{Z(d,\ell)}\Hatf\sphHarm(x),
\end{equation}
with the Fourier-Legendre coefficient given by
$$\Hatf =\langle f,\sphHarm\rangle= \int_{\bbS^d} f(x) \overline{\sphHarm}(x)\dmu.$$
For $\sigma \geq 0$, the Sobolev space $\Hsig=\Hsig(\bbS^d)$ is defined as 
$$H^{\sigma}(\bbS^d)=\{f\in L_2(\bbS^d):\|f\|_{\Hsig(\bbS^d)}<\infty\}$$
with the norm given by
\begin{equation*}\label{eq:Sobolev_norm}
	\|f\|_{\Hsig(\bbS^d)}^2 = \sum_{\ell=0}^{\infty} \sum_{k=1}^{Z(d,\ell)} (1 + \ell)^{2\sigma} |\Hatf|^2.
\end{equation*}

\subsection{Scaled zonal kernels}
Zonal functions on the sphere $\bbS^d$ can be represented as $\varphi(x\cdot y)$, where $x,y\in\bbS^d$ and $\varphi(t)$ is a continuous function on $[-1,1]$, 
which admits the following symmetric representation
\begin{equation}\label{eq:zonalfunctions}
    \varphi(x\cdot y)=\sum_{\ell=0}^{\infty} a_{\ell}P_{\ell}(d+1;x\cdot y), 
\end{equation}
 where $P_{\ell}(d+1;t)$ is the $(d+1)$-dimensional Legendre polynomial of degree $\ell$, normalized such that $P_{\ell}(d+1;1)=1$. 
The series on the right hand side of \eqref{eq:zonalfunctions} is referred to as the ``Fourier-Legendre expansion" of the zonal kernel $\varphi$, and $a_\ell$ the ``Fourier-Legendre coefficients". 
Convergence of the series is in the sense of Schwartz class distributions. 

An efficient way of calculating Fourier-Legendre coefficients is via Funk-Hecke formula, which states that for every spherical harmonic of degree $\ell$ and order $k$, there holds that
\begin{equation}\label{eq:FunkHeckeFormula}
 \int_{\mathbb{S}^d}\varphi(x\cdot y)\sphHarm(y)\d\mu(y)=\widehat{\varphi}(\ell)\sphHarm(x),~
\text{with} ~~\widehat{\varphi}(\ell)=\int_{-1}^1\varphi(t)P_{\ell}(d+1;t)\d t.
  \end{equation}
  This together with the addition formula for spherical harmonics (e.g. \cite[Page 10]{muller1966SphHarm})
leads to
\begin{equation}\label{kernel-coefficient}
    \varphi(x\cdot y)=\sum_{\ell=0}^{\infty}\widehat{\varphi}(\ell)K_{\ell}(x,y),
\end{equation}
in which $\widehat{\varphi}(\ell)=\frac{\omega_d}{Z(d,\ell)}a_{\ell}$, where $\omega_{d}$ is the volume of $\bbS^d$, and \begin{equation}\label{eq:AddiThm}
    K_{\ell}(x,y)=\sum_{k=1}^{Z(d,\ell)}\sphHarm(x)\sphHarm(y)=\frac{Z(d,\ell)}{\omega_d}P_{\ell}(d+1;x\cdot y).
\end{equation}   

Let $\Phi$ be a radial function defined by $\Phi(x)=\phi(\|x\|)$ on $\bbR^{d+1}$ with $\phi\in C[0,\infty]$. For $0<\rho<1$, we define its scaled version as $\Phi_{\rho}(x)=\Phi(\rho^{-1}x)$. By restricting the scaled radial kernel $\Phi_{\rho}(x-y)$ to the sphere, we define the following scaled zonal kernel:
\begin{equation}\label{eq:zonal_kernel}
	\scaleKer(x\cdot y)=\frac{1}{\Lambda_{\phi,\rho}}\Phi_{\rho}(x-y),
\end{equation}
where
\begin{equation}\label{eq:Lambda}
    \Lambda_{\phi,\rho}=\int_{\bbS^d}\Phi_{\rho}(x-y)\dmu,~~x,y\in\bbS^d.
\end{equation}

Similarly, the Fourier-Legendre representation of a zonal kernel is given by
\begin{equation*}\label{eq:ScalerKer}
	\scaleKer(x\cdot y)=\sum_{\ell=0}^{\infty} \widehat{\scaleKer}(\ell)\sum_{k=1}^{Z(d,\ell)}\sphHarm(x)\sphHarm(y).
\end{equation*}

To use this zonal kernel in constructing quasi-interpolation approximations in the next section, we make the following assumption:
\begin{assumption}\label{assump:kernel}
	For the scaled zonal kernel $\scaleKer$ defined in \eqref{eq:zonal_kernel}, we assume that
	\begin{enumerate}[label=(\arabic*)]
		\item For $0<\rho<\rho_0<1$, there exist two constants $m\geq 0$ and $c_1>0$, independent of $\ell,\rho$, such that
        \begin{equation}\label{eq:assump_eq1}
            |1-\widehat{\scaleKer}(\ell)|\leq c_1 \ell^m\rho^m,~~ 0\leq\ell\leq \ell_{\rho}=\lfloor \rho^{-1}-1\rfloor;~~
        \end{equation}
		\item For $\ell>\ell_{\rho}$, there exists a constant $c_2>0$ such that $\widehat{\scaleKer}(\ell)\leq c_2$.
	\end{enumerate}
\end{assumption}
There is a large family of zonal kernels satisfying the above assumption, as verified in \cite{sun-gao-sun}, including the spherical versions of the Poisson kernel, Gaussian kernel, and compactly-supported kernels. In the remainder of this paper, we use $c_1, c_2,...$ and $C_1,C_2,\ldots$ for specific constants, while
$C$ denotes a generic constant, which may take different values in each occurrence. To further facilitate our error analysis, we impose the following decay condition on Fourier-Legendre coefficients of the kernel:
\begin{assumption}\label{assump2}
For the scaled zonal kernel $\scaleKer$, there exist constants $0<c_3\leq c_4<\infty$ such that
    \begin{equation}\label{eq:kernelFourierCond}
    c_3(1+\rho\ell)^{-2s}\leq \widehat{\scaleKer}(\ell)\leq c_4(1+\rho\ell)^{-2s},~~\ell>0.
\end{equation}
\end{assumption}

A scaled zonal kernel $\varphi_{\rho}$ induces its native space 
$$\mathcal{N}_{\Pro}:=\{f\in \mathcal{S}'(\bbS^d):\|f\|_{\Pro}< \infty\}$$
with norm
$$\|f\|_{\Phi_{\rho}}^2=\sum_{\ell=0}^{\infty}\sum_{k=1}^{Z(d,\ell)}\frac{|\widehat{f}_{\ell,k}|^2}{\widehat{\scaleKer}(\ell)}.$$
Here $\mathcal{S}'(\bbS^d)$ denotes the space of Schwartz class distributions on $\bbS^d$.
Le Gia et al. \cite[Lemma~3.1]{gia_2010SINUM_multiscale} proved the following results.
\begin{lemma}\label{lem:Equiv_Sobolev_Native}
    For $\rho\leq 1$ and all $f\in \Hsig(\bbS^d)$, we have the following inequalities
    $$\|f\|_{\Pro}\leq c_3^{-1/2}\|f\|_{\Hsig}\leq (c_4/c_3)^{1/2}\|f\|_{\Phi_1},~~\|f\|_{\Phi_1}\leq (c_4/c_3)^{1/2}\rho^{-s}\|f\|_{\Pro}.$$
\end{lemma}

Scaled kernels satisfying \cref{assump2} have been thoroughly investigated in the literature \cite{gia_2010SINUM_multiscale,legia_2012ACHA_multiscale}. In \cite{gia_2010SINUM_multiscale}, authors showed that the spherical versions of Wendland's kernels have Fourier-Legendre coefficients satisfying the decay condition \eqref{eq:kernelFourierCond}. 
We will show below that Wendland's kernels in even dimensions satisfy \cref{assump:kernel}.
To start, we recall that scaled Wendland's kernels are given by
$$\Phi(x;\rho)=\phi_{l,k}(\|x\|;\rho)=p_k(\rho r)(1-\rho^{-1}\|x\|)^{l+k}_{+},~~x\in \bbR^{d+1},~0<\rho<1,$$
where $p_k$ is a polynomial of degree $k$. For more details on these kernels, we refer to \cite{wendland2004scattered}. In the case of even dimensions $d=2n+2$ with $n\in \bbN_0$, we define the scaled zonal Wendland's kernel by
\begin{equation}\label{eq:restrictedKernel_Wendland}
    \varphi_\rho^w(x\cdot y):=\Lambda_{\rho,\mu,d}^{-1}\phi_{l,k}(\|x-y\|;\rho)|_{x,y\in \bbS^d},
\end{equation}
where 
$$\Lambda_{\rho,\mu,d}=(2\pi)^{\frac{d-1}{2}}C_{\mu} \rho^d,~~\mu=n+k+2,~C_{\mu}=\frac{2^{\mu-\frac{1}{2}}\Gamma(\mu-\frac{1}{2})\Gamma(\mu+1)}{\Gamma(3\mu-1)}.$$

The Fourier-Legendre coefficients of scaled zonal Wendland's kernels were explicitly derived in \cite[Theorem 4.7]{hubbert-2023Adv-generalised}. With this result, we can show the following result.
\begin{lemma}
     Let $d=2n+2$ with $n\in\bbN_0$. Then, for sufficiently small $\rho$, the scaled zonal Wendland's kernel $\varphi_\rho^w(x\cdot y)$ satisfies \cref{assump:kernel} with $m=2$.
\end{lemma}
\begin{proof}
   Changing the ambient space from $\bbR^d$ to $\bbR^{d+1}$, we quote Hubbert and J\"{a}ger's formula \cite{hubbert-2023Adv-generalised}  as
\begin{align*}
    \widehat{\varphi_\rho^w}(\ell)= \hgeom{3}{2}\left[\begin{array}{c} -(\ell+n), \ell+n+1, \mu-\frac{1}{2} \\ \frac{3\mu-1}{2}, \frac{3\mu}{2} \end{array}; \frac{\rho^2}{4} \right],
\end{align*}
where $\hgeom{3}{2}$ denotes the hypergeometric function. 
This reduces to a finite sum as follows:
\begin{align*}
    \widehat{\varphi_\rho^w}(\ell)=\sum_{j=0}^{\ell+n}\frac{\big(-(\ell+n)\big)_j(\ell+n+1)_j(\mu-\frac{1}{2})_j}{(\frac{3\mu-1}{2})_j(\frac{3\mu}{2})_j}\Big(\frac{\rho^2}{4}\Big)^j
    = 1+\sum_{j=1}^{\ell+n}a_{j}(\ell,n)\rho^{2j},
\end{align*}
where $(\cdot)_j$ denotes the Pochhammer symbol, with $(\cdot)_0=1$.  This observation allows us to show that the kernel satisfies the first formula \eqref{eq:assump_eq1} in \cref{assump:kernel} for fixed $\ell$ and sufficiently small $\rho$; see \cite[Prop.~3.7]{sun-gao-sun} for details. The second condition (2) in \cref{assump:kernel} can be  verified directly from \cite[Theorem 6.2]{gia_2010SINUM_multiscale}.
\end{proof}


We end this section with $L_\infty$ and $L_2$-norm estimates of the scaled zonal kernel, which will be utilized in the subsequent sections.   
\begin{lemma}\label{lem:ScaleKernelNorm}
     Let $\phi\in C[0,\infty]$  satisfy  $\int_{0}^{\infty}\phi(z)z^{d}\d z<\infty$, and let $\scaleKer$ be the scaled zonal kernel defined in \eqref{eq:zonal_kernel}. Then for sufficiently small $\rho$ satisfying $0<\rho<\rho_0<1$, there exists a constant $C>0$ independent of $\rho$ such that 
    $$\|\scaleKer\|_{L_{\infty}}\leq C\rho^{-d}\quad \text{and}\quad \|\scaleKer\|_{L_{2}}\leq C\rho^{-d/2}.$$
\end{lemma}
\begin{proof}
First we note that $\Lambda_{\phi,\rho}$ is independent of $x$. We then use spherical coordinates to get
    \begin{align*}\Lambda_{\phi,\rho}=\int_{\bbS^d}\phi(\rho^{-1}\|y-x\|)\d\mu(y)
=\omega_d\int_{0}^{\pi}\phi\big(2\rho^{-1}\sin(\theta/2)\big)\sin^{d-1}(\theta)\d\theta,
\end{align*}
where $\theta$ is the smaller angle between the two vectors $x$ and $y$.
The substitution $z=2\rho^{-1}\sin(\theta/2)$ yields
\begin{align*}
    \Lambda_{\phi,\rho}=&~\omega_d\rho^d\int_{0}^{2/\rho}\phi(z)z^{d-1}\Big(1-\frac{\rho^2z^2}{4}\Big)^{d/2}\d z\\
    =&~\omega_d\rho^d\Big[\int_{0}^{2/\rho}\phi(z)z^{d-1}\d z+\calR_{\rho}\Big],
\end{align*}
where
$$\calR_{\rho}:=\int_{0}^{2/\rho}\phi(z)z^{d-1}\Big[\Big(1-\frac{\rho^2z^2}{4}\Big)^{d/2}-1\Big]\d z.$$
We can bound $\calR_{\rho}$ by
\begin{align*}
    |\calR_\rho|\leq &~ \int_{0}^{2/\rho}\phi(z)z^{d-1}\Big|\Big(1-\frac{\rho^2z^2}{4}\Big)^{d/2}-1\Big|\d z\\
    \leq &~\frac{d}{8}\rho^2\int_{0}^{2/\rho}\phi(z)z^{d+1}\d z
    \leq \frac{d}{4}\rho\int_{0}^{2/\rho}\phi(z)z^d\d z.
\end{align*}
Here we have used the inequality 
$|1-(1-x)^{d/2}|\leq \frac{d}{2}x,~~0<x<1$, 
and $\frac{1}{4}\rho^2 z^2\leq 1$ for $z\in[0,2/\rho]$.
Hence, we have 
$$\Lambda_{\phi,\rho}\geq \frac{1}{2}\omega_d\rho^d\int_{0}^{2/\rho}\phi(z)z^{d-1}\d z,$$
which leads to the uniform bound
\begin{align*}\label{eq:LinfBound}
    \|\scaleKer\|_{L_{\infty}}\leq \Lambda_{\phi,\rho}^{-1}\phi(0)\leq C\rho^{-d}.
\end{align*}
 The proof is completed by noting $\|\scaleKer\|_{L_2}\leq \|\scaleKer\|_{L_{\infty}}^{1/2}\|\scaleKer\|_{L_1}^{1/2}$ and $\|\scaleKer\|_{L_1}=1$.
\end{proof}

\section{Deterministic quasi-interpolation on the sphere}
\label{sec:sph_qi}
In this section, we discuss scaled kernel quasi-interpolation from a deterministic perspective. The whole process will be carried out in two steps, which are exposited respectively in \cref{subsec:conv_approx} and \cref{subsec:QMC_method}. Specifically, in \cref{subsec:conv_approx}, we employ spherical convolution operators to approximate target functions under \cref{assump:kernel}. In \cref{subsec:QMC_method},  we discretize the convolution integrals using a simple quasi-Monte Carlo method. In \cref{subsec:MultilevelQI}, we use a multilevel quasi-interpolation technique to accelerate the convergence.


\subsection{Convolution approximation}
\label{subsec:conv_approx}

Let $\scaleKer$ be a scaled zonal kernel defined in \eqref{eq:zonal_kernel}. Given a function $f\in\Hsig(\bbS^d)$, we define the convolution operator $\mathcal{C}_{\scaleKer}$ by
\begin{equation}\label{eq:ConvOper}
    \mathcal{C}_{\scaleKer}:f\mapsto  f*\scaleKer~~\text{with} ~~f*\scaleKer(x)=\int_{\bbS^d}f(y)\scaleKer(x\cdot y)\d\mu(y).
\end{equation}

  The following Sobolev error estimate for the convolution operator was established in \cite[Thm.~3.1]{sun-gao-sun}.
\begin{lemma}\label{lem:L2ConvErr}
	Let $f\in\Hsig(\bbS^d)$ with $\sigma>d/2$. Suppose that the scaled zonal kernel $\scaleKer$ satisfies  \cref{assump:kernel} with $0<\rho<\rho_0<1$, and $0\leq \tau\leq m\leq \sigma$. Let $\mathcal{C}_{\scaleKer}$ be the convolution operator defined in \eqref{eq:ConvOper}. Then, there exists a constant $C$ independent of $\rho$ such that
	$$\|f-\mathcal{C}_{\scaleKer}f\|_{H^{\tau}}\leq C\rho^{m-\tau}\|f\|_{\Hsig}.$$
\end{lemma}

Next, we derive an uniform error bound for the convolution operators in terms of the scaling parameter $\rho$, which we will use frequently in the subsequent analysis.

\begin{theorem}\label{thm:ConvErr}
	Let $f\in\Hsig(\bbS^d)$ with $\sigma\geq m>d/2$. Suppose $\scaleKer$ satisfies \cref{assump:kernel} with $0<\rho<\rho_0<1$. Then there exists a constant $C>0$ independent of $\rho$ such that
	$$\|f-\mathcal{C}_{\scaleKer}f\|_{L_{\infty}}\leq C\rho^{m-d/2}\|f\|_{\Hsig}.$$
\end{theorem}

\begin{proof}
	For each fix $x \in \mathbb S^d$, we have	$$|f(x)-\mathcal{C}_{\scaleKer}f(x)|\leq\sum_{\ell=0}^{\infty}\sum_{k=1}^{Z(d,\ell)}\big|1-\widehat{\scaleKer}(\ell)\big|\big|\Hatf\sphHarm(x)\big|.
	$$
We apply the Cauchy-Schwarz inequality to get
	\begin{equation}\label{eq:linferrorineq1}
		\begin{aligned}
			|f(x)-\mathcal{C}_{\scaleKer}f(x)|\leq&~
  \Big(\sum_{\ell=0}^{\infty}\sum_{k=1}^{Z(d,\ell)}|1-\widehat{\scaleKer}(\ell)|^2|\sphHarm(x)|^2 (1+\ell)^{-2\sigma}\Big)^{1/2}\\
&\quad\cdot\Big(\sum_{\ell=0}^{\infty}\sum_{k=1}^{Z(d,\ell)}(1+\ell)^{2\sigma}|\Hatf|^2\Big)^{1/2}\\\leq&~	
             ||f||_{H^{\sigma}}\cdot\Big(\sum_{\ell=0}^{\infty}|1-\widehat{\scaleKer}(\ell)|^2(1+\ell)^{-2\sigma}\sum_{k=1}^{Z(d,\ell)}|\sphHarm(x)|^2\Big) ^{1/2} \\  
             = &~\omega_d^{-1/2}||f||_{H^{\sigma}}\cdot\Big(\sum_{\ell=0}^{\infty}|1-\widehat{\scaleKer}(\ell)|^2Z(d,\ell)(1+\ell)^{-2\sigma}\Big) ^{1/2}.
		\end{aligned}
	\end{equation}
Here we have used the identity
\[
\sum_{k=1}^{Z(d,\ell)}|\sphHarm(x)|^2 = \frac{Z(d,\ell)}{\omega_d},
\]
which can be derived from the summation formula for spherical harmonics; see \eqref{eq:AddiThm}.
   	In the remainder of the proof, we estimate the last sum in \eqref{eq:linferrorineq1}. We break it into two parts indexed by $\ell\leq \ell_{\rho}$ and $\ell>\ell_{\rho}$, respectively. For the first part $\ell\leq \ell_{\rho}$, we use  \cref{assump:kernel} to have $|1-\widehat{\scaleKer}(\ell)|\leq c_1(\rho\ell)^m$ for $\ell\leq \ell_{\rho}$. We then use the relation $\sum_{\ell=0}^{L}Z(d,\ell)=Z(d+1,L)$ from \cite{muller1966SphHarm} to bound the sum as follows:
	\begin{equation*}
		\begin{aligned}
			\sum_{\ell\leq \ell_{\rho}}|1-\widehat{\scaleKer}(\ell)|^2Z(d,\ell)(1+\ell)^{-2\sigma}
			\leq&~ c_1^2\rho^{2m}\sum_{\ell\leq \ell_{\rho}}\ell^{2m}Z(d,\ell)(1+\ell)^{-2\sigma}\\
			 \leq &~c_1^2\rho^{2m}\sum_{\ell\leq \ell_{\rho}}Z(d,\ell)
			= c_1^2\rho^{2m}Z(d+1,\ell_{\rho})\\
			\leq&~  c_1^2\rho^{2m}C(1+\ell_{\rho})^{d}\leq c_1^2C\rho^{2m-d},
		\end{aligned}
	\end{equation*}
	where we have used $m\leq \sigma$, $\ell^{2m}(1+\ell)^{-2\sigma}<1$ and $Z(d,\ell)\sim (1+\ell)^{d-1}$.
	
	For the second sum $\ell>\ell_{\rho}$, we use the uniform bound $|\widehat{\scaleKer}(\ell)|\leq c_2$ from Condition (2) in \cref{assump:kernel}. This yields
	\begin{equation*}
		\begin{aligned}
			\sum_{\ell>\ell_{\rho}}|1-\widehat{\scaleKer}(\ell)|^2Z(d,\ell)(1+\ell)^{-2\sigma}
			\leq &~ (1+c_2)^2\sum_{\ell>\ell_{\rho}}Z(d,\ell)(1+\ell)^{-2\sigma}\\
			\leq &~ C(1+c_2)^2\sum_{\ell>\ell_{\rho}}(1+\ell)^{d-1}(1+\ell)^{-2\sigma}\\
			\leq &~ C(1+c_2)^2 \ell_{\rho}^{-2\sigma+d}\leq C(1+c_2)^2\rho^{2\sigma-d}.
		\end{aligned}
	\end{equation*}
	Combining the two estimates and noting that $m\leq \sigma$, we conclude that
	$$|f(x)-\mathcal{C}_{\scaleKer}f(x)|\leq C\omega_d^{-1/2}\Big(c_1\rho^{m-d/2}+(1+c_2)\rho^{\sigma-d/2}\Big)\|f\|_{\Hsig},$$
	which completes the proof.
\end{proof}

\begin{corollary}\label{corol:GeneralLp_err}
    Let $f\in\Hsig(\bbS^d)$ with $\sigma\geq m>d/2$. Suppose $\scaleKer$ satisfies \cref{assump:kernel} with $0<\rho<\rho_0<1$. Then for $2\leq p\leq\infty$, there exists a constant $C>0$ independent of $\rho$ such that
	$$\|f-\mathcal{C}_{\scaleKer}f\|_{L_p}\leq C\rho^{m-d\big(\frac{1}{2}-\frac{1}{p}\big)}\|f\|_{\Hsig}.$$
\end{corollary}
\begin{proof}
    We use the interpolation inequality to write
    $$\|f-\mathcal{C}_{\scaleKer}f\|_{L_p}\leq \|f-\mathcal{C}_{\scaleKer}f\|_{L_2}^{2/p}\|f-\mathcal{C}_{\scaleKer}f\|_{L_{\infty}}^{1-2/p}.$$
    This together with the estimates in  \cref{lem:L2ConvErr} and \cref{thm:ConvErr} leads to
    \begin{align*}
        \|f-\mathcal{C}_{\scaleKer}f\|_{L_p}\leq C\big(\rho^{m}\|f\|_{\Hsig}\big)^{2/p}\cdot \big(\rho^{m-d/2}\|f\|_{\Hsig}\big)^{1-2/p}
        =C\rho^{m-d\big(\frac{1}{2}-\frac{1}{p}\big)}\|f\|_{\Hsig},
    \end{align*}
    which is the desired result.
\end{proof}

To construct the final scaled kernel quasi-interpolant, we discretize the spherical convolution integral in \eqref{eq:ConvOper}. Let $X=\{x_1,\ldots,x_N\}\subset \bbS^d$ be a set of points on the sphere. Using the spherical quadrature rules $\{x_j,\alpha_j\}_{j=1}^{N}$, where $\alpha_j$ are the quadrature weights, we define the spherical quasi-interpolant as
\begin{equation}\label{eq:SphQi}
	\cQ_{\scaleKer} f(x)=\sum_{j=1}^ {N}\alpha_jf(x_j)\scaleKer(x\cdot x_j).
\end{equation}

Authors of \cite{sun-gao-sun}  employed quadrature rules of high order to discretize the convolution integrals, which inevitably suffers from computational instability. In this paper, we broaden the method’s applicability in two new directions. First, we take a deterministic approach based on quasi-Monte Carlo quadrature rules from \cite{brauchart-2014MCoM-qmc}.
Second, we develop a stochastic discretization using the Monte Carlo integration method, which we will analyze in \cref{sec:Sto_qi_sph}.

\subsection{Quasi-Monte Carlo method}
\label{subsec:QMC_method}
A point set $X^q:=\{x_j\}_{j=1}^{N}$ on $\bbS^d$ is called a sequence of QMC designs \cite{brauchart-2014MCoM-qmc} for $H^{\sigma}(\bbS^d)$ ($\sigma>d/2$), if there exists a constant $C(\sigma,d)>0$ independent of $N$ such that
\begin{equation}\label{eq:QMC_estimate}
	\sup_{f\in H^{\sigma}(\bbS^d)}\Big|\frac{1}{N}\sum_{j=1}^{N}f(x_j)-\IntSph f(x)\d\mu(x)\Big|\leq C(\sigma,d)N^{-\sigma/d}\fHsig.
\end{equation}
Using such QMC designs, we construct the QMC quasi-interpolant as
\begin{equation}\label{eq:QMCQI}
	\cQ_{\scaleKer}^qf=\frac{1}{N}\sum_{j=1}^{N}f(x_j)\scaleKer(x\cdot x_j),~~x_j\in X^q,~x\in\bbS^d.
\end{equation}
 
\begin{theorem}\label{thm:QMC_error}
	Let $f\in \Hsig(\bbS^d)$ with $\sigma>d/2$, and let $\cQ_{\scaleKer}^q f$ be defined in \eqref{eq:QMCQI} with a scaled zonal kernel $\scaleKer$ satisfying \cref{assump:kernel} with $0<\rho<\rho_0<1$, $m\leq\sigma$ and \cref{assump2} with $s=\tau+\sigma$ for $0\leq \tau\leq m$. If $X^q$ is a QMC design for $\Hsig(\bbS^d)$, then there exists a constant $C>0$ independent of $\rho$ and $N$, such that
	$$\|f-\cQ_{\scaleKer}^q f\|_{H^{\tau}}\leq C\Big(\rho^{m-\tau}  +\rho^{-s}N^{-\sigma/d}\Big)\|f\|_{\Hsig}.$$
\end{theorem}
\begin{proof}
	We write
	\begin{equation}\label{eq:SobolevQMCerr1}
		\|f-\cQ_{\scaleKer}^q f\|_{H^{\tau}}\leq \|f-\mathcal{C}_{\scaleKer}f\|_{H^{\tau}}+\|\mathcal{C}_{\scaleKer}f-\cQ_{\scaleKer}^q f\|_{H^{\tau}}.
	\end{equation}
	A straightforward application of \cref{lem:L2ConvErr} results in the bound:
    \begin{equation}\label{eq:qmc_conv_error}
        \|f-\mathcal{C}_{\scaleKer}f\|_{H^{\tau}}\leq C \rho^{m-\tau} \|f\|_{\Hsig}.
    \end{equation}
	To estimate the second term at the right-hand side of \eqref{eq:SobolevQMCerr1}, we make use of its Fourier-Legendre expansion
\begin{equation*}\label{eq:ErrDefinition}
		\begin{aligned}
			\|f*\scaleKer-\cQ_{\scaleKer}^q f\|_{H^{\tau}}^2=&~\sum_{\ell=0}^{\infty}\sum_{k=1}^{Z(d,\ell)}(1+\ell)^{2\tau}\big|\Hatf\HatKer-\HatQi\big|^2,
		\end{aligned} 
	\end{equation*}
	where 
	\begin{equation*}
		\begin{aligned}
			\HatQi 
			=~\widehat{\scaleKer}(\ell)\frac{1}{N}\sum_{j=1}^Nf(x_j)Y_{\ell,k}(x_j).
		\end{aligned}
	\end{equation*}
    This gives
	\begin{equation*}\label{eq:IntErrSobolevnorm}
		\begin{aligned}
			\|f*\scaleKer-\cQ_{\scaleKer}^q f\|_{\Htau}^2=\sum_{\ell=0}^{\infty}\sum_{k=1}^{Z(d,\ell)}(1+\ell)^{2\tau}|\widehat{\scaleKer}(\ell)|^2\cdot\big|\widehat{f}_{\ell,k}-\frac{1}{N}\sum_{j}f(x_j)\sphHarm(x_j)\big|^2.
		\end{aligned}
	\end{equation*}
	
    By \eqref{eq:QMC_estimate}, the QMC quadrature error for each Fourier-Legendre coefficient satisfies
	\begin{equation*}
		\begin{aligned}
			\big|\widehat{f}_{\ell,k}-\frac{1}{N}\sum_{j=1}^Nf(x_j)\sphHarm(x_j)\big|
			\leq  ~C(\sigma,d)N^{-\sigma/d}\|f\|_{H^{\sigma}}\|\sphHarm\|_{H^{\sigma}}.
		\end{aligned}
	\end{equation*}
	Using the Sobolev norm relation for spherical harmonics
	$\|\mathcal{Y}_{\ell,k}\|_{\Hsig}^2=(1+\ell)^{2\sigma},$
	we get
	\begin{equation*}
		\begin{aligned}
			&~\|f*\scaleKer-\cQ_{\scaleKer}^q f\|_{\Htau}^2 \\ \leq&~C^2(\sigma,d)N^{-2\sigma/d}\|f\|_{\Hsig}^2\sum_{\ell=0}^{\infty}\sum_{k=1}^{Z(d,\ell)}(1+\ell)^{2(\tau+\sigma)}|\widehat{\scaleKer}(\ell)|^2 \\
            =&~C^2(\sigma,d)N^{-2\sigma/d}\|f\|_{\Hsig}^2\|\scaleKer\|_{H^s}^2\\
			\leq &~C^2(\sigma,d)N^{-2\sigma/d}c_4^{1/2}(c_4/c_3)^{1/2}\rho^{-2s}\|f\|_{\Hsig}^2\|\scaleKer\|_{\Phi_{\rho}}^2,
		\end{aligned}
	\end{equation*}
    where the last inequality follows from \cref{lem:Equiv_Sobolev_Native}.
	This completes the proof.
\end{proof}

\begin{corollary}\label{corol:QMCError}
	Suppose that the assumptions of \cref{thm:QMC_error} hold, Pick $\rho=\mathcal{O}(N^{-\frac{\sigma}{(\sigma+m)d}})$, then there exists a constant $C>0$ independent of $\rho$ and $N$ such that
	$$\|f-\cQ_{\scaleKer}^q f\|_{H^{\tau}}\leq C N^{-\frac{(m-\tau)\sigma}{(m+\sigma)d}}\|f\|_{\Hsig}.$$
\end{corollary}

\subsection{Multilevel quasi-interpolation on the sphere}
\label{subsec:MultilevelQI}
In this section, we present a multilevel quasi-interpolation scheme to approximate functions defined on the unit sphere $\bbS^d$. 
Our approach is realized via an error correction transform, where the final approximant is constructed by combining quasi-interpolants of graded levels. Compared with the multiscale scheme in the literature \cite{gia_2010SINUM_multiscale,legia_2012ACHA_multiscale,sharon-2023SISC-multiscale}, a key advantage of our approach is that it does not require the solution of any linear system, which can be computationally efficient for large-scale problems.

Our approach is based on the construction of a sequence of quasi-uniform point QMC sets $X_1,X_2,\ldots$ on the sphere $\bbS^d$, where the fill distance $h_j$ of each set $X_j$ satisfies the relation $h_{j+1} \approx \delta h_j$ for some fixed factor $\delta \in (0,1)$. This ensures a uniform decrease in the fill distance as the level $j$ increases. The quasi-uniformity of the point sets is characterized by the existence of a constant $c_q \geq 1$ such that the separation distance $q_j$ and the fill distance $h_j$ are related by $q_j \leq h_j \leq c_q q_j$ and $|X_j|=\mathcal{O}(h_j^{-d})$.  

The quasi-interpolation operator $Q_{X_j, \rho_j}$ for the $j$-th scaled kernel $\varphi_{\rho_j}$ is defined as
$$Q_{X_j,\rho_j}f(x)=\frac{1}{|X_j|}\sum_{x_k\in X_j}f(x_k)\varphi_{\rho_j}(x\cdot x_k).$$

The idea behind the multilevel scheme is a simple residual correction method. We start by defining two sequences of operators $\cM_j$ and $\cE_j$, where $\cE_j$ describes the error at level $j$ and $\cM_j$ represents the multilevel approximation at level $j.$  By setting $\cM_0f=0$ and $\cE_0f=f$ and for $j=1,2,\ldots$, we can compute
$$\cM_jf=\cM_{j-1}f+Q_{X_j,\rho_j}\cE_{j-1}f,$$
$$\cE_jf=\cE_{j-1}f-Q_{X_j,\rho_j}\cE_{j-1}f.$$

Building upon the multilevel framework developed by Franz and Wendland \cite{franz-wendland2023multilevel} for $\bbR^d$, we now establish analogous recursive relations for spherical quasi-interpolation on $\bbS^d$.
\begin{proposition}
	The approximation and error operators of the spherical multiscale quasi-interpolation scheme satisfy the following recursive relations
	$$\cM_n=\sum_{j=1}^nQ_{X_j,\rho_j}\prod_{\ell=1}^{j-1}\big(I-Q_{X_{\ell},\rho_{\ell}}\big),~~\cE_n=\prod_{j=1}^n\big(I-Q_{X_j,\rho_j}\big),$$
	where $I$ denotes the identity operator.
\end{proposition}
\begin{proof}
	We establish both identities by mathematical induction. The base case $\cE_1=I-Q_{X_1,\rho_1}$ follows immediately from the definition. The inductive step follows from the telescoping relation
	$$\cE_{n+1}=(I-Q_{X_{n+1},\rho_{n+1}})\cE_n,$$
	which directly yields the product formula. 
	
	For the approximation operators $\cM_n$, we have 
	$$\cM_1=\cM_0+Q_{X_1,\rho_1}\cE_0=Q_{X_1,\rho_1}.$$
	Assuming the formula holds for level $n$, the update rule
	$$\cM_{n+1}=\cM_{n}+Q_{X_{n+1},\rho_{n+1}}\cE_{n}$$
	combines with the inductive hypothesis to give
	\begin{equation*}
		\cM_{n+1}=\sum_{j=1}^nQ_{X_j,\rho_j}\prod_{\ell=1}^{j-1}\big(I-Q_{X_{\ell},\rho_{\ell}}\big)+Q_{X_{n+1},\rho_{n+1}}\prod_{j=1}^n\big(I-Q_{X_j,\rho_j}\big),
	\end{equation*}
	which completes the proof.
\end{proof}

The following theorem establishes the convergence properties of the proposed multilevel quasi-interpolation method on the sphere.

\begin{theorem}\label{thm:MultilevelQI}
Let $\{X_j\}_{j=1}^n \subseteq \mathbb{S}^d$ be a sequence of QMC designs with fill-distances $\{h_j\}_{j=1}^n$ satisfying $c_{\delta}\delta h_j\leq h_{j+1}\leq \delta h_j$ for some $\delta \in (0,1)$ and $0<c_{\delta}\leq 1$. Let the scale parameters be chosen as $\rho_j = \nu h_j^{\scriptscriptstyle 1/2}$ with $\nu > 1$. Assume that the kernel $\scaleKer$ satisfies \cref{assump:kernel} with $m=\sigma$ and \cref{assump2} with $s = 2\sigma$. Let the multilevel quasi-interpolation approximation $\cM_nf$ be defined as in \cref{alg:multilevel_quasi_interpolation}. Then, there exists a constant $C>0$ such that
\begin{equation*}
\|f - \cM_nf\|_{L_2} \leq C h_1^{\sigma/2}\beta^{n-1} \|f\|_{\Hsig},
\end{equation*}
where $\beta = C_{\nu}\delta^{\sigma/2}$ with $C_{\nu}=1+\nu^{-2\sigma}$.
\end{theorem}
 
 \begin{proof}
Since the cardinality of the point sets satisfies $|X_j|=N_j\leq Ch_j^{-d}$ and $\rho_j=\nu h_{j}^{1/2}$, we can apply \cref{thm:QMC_error} with $\tau=0$ and $\tau=\sigma$, respectively, to obtain the following two estimates:
\begin{equation}\label{eq:multilevel_err1}
\begin{aligned}
    \|\cE_{j-1}f-Q_{X_j,\rho_j}\cE_{j-1}f\|_{L_2}\leq &~C_1\big(\rho_j^{\sigma}+h_j^{\sigma}\rho_j^{-\sigma}\big)\|\cE_{j-1}f\|_{H^{\sigma}}\\
    \leq &~C_1(\nu^{\sigma}+\nu^{-\sigma})h_j^{\sigma/2}\|\cE_{j-1}f\|_{H^{\sigma}},
    \end{aligned}
\end{equation}
and
\begin{equation}\label{eq:multilevel_err2}
\begin{aligned}
    \|\cE_{j-1}f-Q_{X_j,\rho_j}\cE_{j-1}f\|_{H^{\sigma}}\leq &~ C_2(1+h_j\rho_j^{-2\sigma})\|\cE_{j-1}f\|_{\Hsig}\\
    \leq &~ C_2(1+\nu^{-2\sigma})\|\cE_{j-1}f\|_{\Hsig}.
    \end{aligned}
\end{equation}
By \eqref{eq:multilevel_err2}, we further derive
\begin{align}\label{eq:multilevel_err3}
    \|\cE_j f\|_{\Hsig}
    =\|\cE_{j-1}f-Q_{X_j,\rho_j}\cE_{j-1}f\|_{H^{\sigma}}
    \leq  C_2\big(1+\nu^{-2\sigma}\big)\|\cE_{j-1}f\|_{\Hsig}.
\end{align}
Let $C_\nu=1+\nu^{-2\sigma}$, Repeating the inequality \eqref{eq:multilevel_err3} for $j$ times yields
\begin{align}\label{eq:multilevel_err4}
    \|\cE_j f\|_{\Hsig}
    \leq C_2C_{\nu}^j\|\cE_0f\|_{\Hsig}=C_2C_{\nu}^j\|f\|_{\Hsig}.
\end{align}

Now using \eqref{eq:multilevel_err1}, we can estimate the error at level $n$ by
\begin{equation*}
\begin{aligned}
    \|f-\cM_n f\|_{L_2}=&~\|\cE_{n-1}f-Q_{X_n,\rho_n}\cE_{n-1}f\|_{L_2}\\
    \leq&~ C_1(\nu^{\sigma}+\nu^{-\sigma})\ h_n^{\sigma/2}\|\cE_{n-1}f\|_{H^{\sigma}}.
\end{aligned}
\end{equation*}
This together with \eqref{eq:multilevel_err4} and the assumption $h_n\leq\delta h_{n-1}\leq\delta^{n-1} h_1$ gives
\begin{align*}
    \|f-\cM_n f\|_{L_2}\leq&~ C_1C_2(\nu^{\sigma}+\nu^{-\sigma})C_{\nu}^{n-1}\ h_n^{\sigma/2}\|f\|_{H^{\sigma}}\\
\leq &~C_1C_2(\nu^{\sigma}+\nu^{-\sigma})h_1^{\sigma/2}\big(C_{\nu}\delta^{\sigma/2}\big)^{n-1}\|f\|_{H^{\sigma}}.
\end{align*}
Finally, defining $C=C_1C_2(\nu^{\sigma}+\nu^{-\sigma})$, we complete the proof.
\end{proof}

This result demonstrates that the multilevel quasi-interpolation method achieves convergence with respect to the level $n$, provided that the parameters $\nu$ and $\delta$ are selected such that $\beta<1$, and the fill distance $h_1$ is sufficiently small to ensure that the estimate in \cref{thm:QMC_error} holds at the first level.

\begin{algorithm}
\caption{Multilevel quasi-interpolation on the sphere}
\label{alg:multilevel_quasi_interpolation}
\begin{algorithmic}[1]
\Require~~ 
    \begin{itemize}
        \item Target function $f$ defined on the unit sphere $\mathbb{S}^d$
        \item  A sequence of quasi-uniform QMC point sets $\{X_j\}_{j=1}^n$ on $\mathbb{S}^d$ with fill distances $\{h_j\}_{j=1}^n$ satisfying $c_{\delta}\delta h_j\leq h_{j+1}\leq \delta h_j$ for some $\delta \in (0,1)$ and $0<c_{\delta}\leq 1$
        \item Scale parameters $\{\rho_j\}_{j=1}^n$ chosen as $\rho_j = \nu h_j^{1/2}$ for a fixed $\nu > 1$
    \end{itemize}
\Ensure~~
    \begin{itemize}
        \item Final approximation $\cM_n f$ to the target function $f$
    \end{itemize}
\State Initialize $\cM_0 f = 0$ and $\cE_0f = f$
\For{$j = 1, 2, \ldots, n$}
    \State Apply the quasi-interpolation operator $Q_{X_j, \rho_j}$ to $\cE_{j-1}f$ and denote $s_j = Q_{X_j, \rho_j}\cE_{j-1}f$
    \State Update the approximation $\cM_j f = \cM_{j-1}f + s_j$
    \State Update the error $\cE_j f = \cE_{j-1}f - s_j$
\EndFor\\
\Return the final approximation $\cM_n f$
\end{algorithmic}
\end{algorithm}

\section{Stochastic quasi-interpolation on the sphere}
\label{sec:Sto_qi_sph}

Let $\bX$ denote a random variable uniformly distributed on the sphere $\bbS^d$. We designate $\{\bX_j\}_{j=1}^N$ as a collection of $N$ independent and identically distributed (i.i.d.) copies of $\bX$, and let $\{f(\bX_j)\}_{j=1}^N$ represent the discrete evaluations of a continuous target function $f$ at these random sample points. We approximate the integral within the convolution operator \eqref{eq:ConvOper} via random sampling, thereby deriving a Monte Carlo quasi-interpolant (MCQI):
\begin{equation}\label{eq:MCQI}
	\cQ_{\scaleKer}^rf(x)=\frac{1}{N}\sum_{j=1}^{N}f(\bX_j)\scaleKer(x\cdot \bX_j),~~x\in \bbS^d.
\end{equation}

In what follows, we establish an $L_2$-probabilistic concentration inequality for MCQI, which exhibits an exponential decay rate. As a direct consequence, we obtain mean $L_2$ and $L_{\infty}$ convergence. Before proceeding, we recall a bounded differences (McDiarmid-type) inequality that will be used in the analysis.
\begin{lemma}[Bounded difference inequality {\cite{mcdiarmid-1989SurveyComb-method}}]
\label{lem:bounded-difference}
Let $\Omega$ be a measurable space and let 
$\cF:\Omega^{N}\to\mathbb R$ be a mapping for which there exist
non–negative constants $a_{1},\dots ,a_{N}$ such that
\begin{equation}\label{eq:bounded-diff}
  \bigl|
    \cF(x_{1},\dots ,x_{j},\dots ,x_{N})-
    \cF(x_{1},\dots ,x_{j}',\dots ,x_{N})
  \bigr|
  \le a_{j},
  \qquad
  1\le j\le N,
\end{equation}
for all points $x_{1},\dots ,x_{N},x_{j}'\in\Omega$.
Let $\bX_{1},\dots ,\bX_{N}$ be independent $\Omega$–valued random variables.
Then, for every $\epsilon>0$,
\begin{equation}\label{eq:mcdiarmid}
  \mathbb P\Bigl(
      \bigl|\cF(\bX_{1},\dots ,\bX_{N})
            -\mathbb E\cF(\bX_{1},\dots ,\bX_{N})\bigr|
      \ge \epsilon
  \Bigr)
  \;\le\;
  2\exp\Bigl(
      -\,\frac{2\epsilon^{2}}{\sum_{j=1}^{N}a_{j}^{2}}
  \Bigr).
\end{equation}
\end{lemma}

\begin{theorem}($L_2$-probabilistic concentration inequality under random sampling)\label{thm:ProbConv-MC}
	Suppose that the scaled zonal kernel $\scaleKer$ in \eqref{eq:zonal_kernel} satisfies \cref{assump:kernel}. Let $f\in \Hsig(\bbS^d)$ with $\sigma\geq m>d/2$. The MCQI $\cQ_{\scaleKer}^rf$ is defined in \eqref{eq:MCQI}. For any $\epsilon>0$, let $\rho$ and $N$ be chosen such that
    $$C\big(\rho^{m}\|f\|_{H^\sigma}+\rho^{-d/2}N^{-1/2}\|f\|_{L_\infty}\big)\leq \frac{\epsilon}{2},$$
    where $C_3$ is a constant given in \eqref{eq:ProbConv_4}.
    Then, there exists a constant $C>0$ independent of $\rho$ and $N$ such that
	$$\bbP\Big\{\|\cQ_{\scaleKer}^r f-f\|_{L_2}\geq\epsilon\Big\}\leq 2\exp\Big(\!-C\rho^d N\epsilon^2\|f\|_{L_{\infty}}^{-2}\Big).$$
\end{theorem}
\begin{proof}
	We define the random variable $\bZ:=\|\cQ_{\scaleKer}^r f-\mathcal{C}_{\scaleKer}f\|_{L_2}$, let $\bbE\bZ$ denote its expectation, and apply the $L_2$ triangle inequality to obtain 
\begin{equation}\label{eq:ProbConv_1}
		\begin{aligned}
			\|\cQ_{\scaleKer}^r f-f\|_{L_2}\leq &~\|\cQ_{\scaleKer}^r f-\mathcal{C}_{\scaleKer}f\|_{L_2}+\|\mathcal{C}_{\scaleKer}f-f\|_{L_2}\\
			\leq &~\big|\bZ-\bbE \bZ\big|+\bbE \bZ+\|\mathcal{C}_{\scaleKer}f-f\|_{L_2}.
		\end{aligned}
	\end{equation}
	The third term $\|\mathcal{C}_{\scaleKer}f-f\|_{L_2}$ on the right-hand side of \eqref{eq:ProbConv_1} corresponds to the convolution approximation error, which can be bounded deterministically by \cref{lem:L2ConvErr}:
	\begin{equation}\label{eq:ProbConv_2}
		\|\mathcal{C}_{\scaleKer}f-f\|_{L_2}\leq C\rho^{m}\|f\|_{H^\sigma}.
	\end{equation}
    
	To estimate the second term in \eqref{eq:ProbConv_1}, we first employ Fubini's theorem to write
	\begin{equation*}
		\begin{aligned}
			\bbE \bZ^2=\int_{\bbS^d}\Big[\bbE\big|\cQ_{\scaleKer}^r f(x)-\mathcal{C}_{\scaleKer}f(x)\big|^2\Big]\dmu.
		\end{aligned} 
	\end{equation*}
	We then introduce an auxiliary random variable $$\bY_{x} := f(\bm{X})\scaleKer(x \cdot \bm{X}),~~x\in\bbS^d.$$
	The second moment of $\bY_x$ is given by
	\begin{equation*}\label{eq:L2toL2LinfError}
		\begin{aligned}
			\bbE[\bY_{x}^2] = &~\int_{\bbS^d} f^2(y)\scaleKer^2(x \cdot y) \d \mu(y)\\
			\leq &~ \|f\|_{L_\infty}^2 \int_{\bbS^d}\scaleKer^2(x \cdot y) \d \mu(y)
			\leq   C\rho^{-d}\|f\|_{L_\infty}^2.
		\end{aligned}
	\end{equation*}
	where the last inequality follows from the $L_2$ bound of the scaled kernel in \cref{lem:ScaleKernelNorm}.
	With this result, we can estimate the mean squared error as
	\begin{equation}\label{eq:MSEerror1}
		\bbE\Big[ \big|   \cQ_{\scaleKer}^r f(x)-\mathcal{C}_{\scaleKer}f(x) \big|^2 \Big] \leq N^{-1}\bbE[\bY_{x}^2] \leq C\rho^{-d}N^{-1}\|f\|_{L_\infty}^2.
	\end{equation}
	Consequently, we can apply Jensen's inequality
	$\bbE \bZ\leq \sqrt{\bbE \bZ^2}$ to obtain
	\begin{equation}\label{eq:ProbConv_3}
		\bbE\bZ=\bbE\Big[\|\cQ_{\scaleKer}^r f-\mathcal{C}_{\scaleKer}f\|_{L_2}\Big]\leq C\rho^{-d/2}N^{-1/2}\|f\|_{L_{\infty}}.
	\end{equation}
	
	By combining the estimates from \eqref{eq:ProbConv_1}, \eqref{eq:ProbConv_2} and \eqref{eq:ProbConv_3}, we arrive at
    \begin{equation}\label{eq:ProbConv_4}
        \|\cQ_{\scaleKer}^r f-f\|_{L_2}\leq\big|\bZ-\bbE \bZ\big|+ C_3\big(\rho^{m}\|f\|_{H^\sigma}+\rho^{-d/2}N^{-1/2}\|f\|_{L_\infty}\big).
    \end{equation}
	Now, under the assumption
	$$C_3\big(\rho^{m}\|f\|_{H^\sigma}+\rho^{-d/2}N^{-1/2}\|f\|_{L_\infty}\big)\leq \frac{\epsilon}{2},$$
	it follows that for $\|\cQ_{\scaleKer}^r f-f\|_{L_2}\geq\epsilon$ to hold, it is necessary that 
	$$\big|\bZ-\bbE \bZ\big|\geq\frac{\epsilon}{2}.$$
	Hence, we obtain the following probability inequality:
	\begin{equation}\label{eq:Prob_Ineq_1}
		\bbP\Big\{\|\cQ_{\scaleKer}^r f-f\|_{L_2}\geq\epsilon\Big\}\leq \bbP\Big\{\big|\bZ-\bbE\bZ\big|\geq\frac{\epsilon}{2}\Big\}.
	\end{equation}
	
	Finally, to estimate the probability on the right-hand side of \eqref{eq:Prob_Ineq_1}, we define an auxiliary function
	$$\cF(x_1,\ldots,x_N):=\big\|\frac{1}{N}\sum_{j=1}^N f(x_j)\scaleKer(x\cdot x_j)-\mathcal{C}_{\scaleKer}f(x)\big\|_{L_2}.$$
	We can use the triangle inequality to estimate the difference
	\begin{equation*}
		\begin{aligned}
			&|\cF(x_1,\ldots,x_j,\ldots,x_N)-\cF(x_1,\ldots,x_j',\ldots,x_N)|\\
			\leq&~\frac{1}{N}\|f(x_j)\scaleKer(x\cdot x_j)-f(x_j')\scaleKer(x\cdot x_j')\|_{L_2}\\
			\leq &~\frac{2}{N}\|f\|_{L_{\infty}}\|\scaleKer\|_{L_2}\leq C\rho^{-d/2}N^{-1}\|f\|_{L_{\infty}}.
		\end{aligned}
	\end{equation*}
	Since $\bZ=\cF(\bX_1,\ldots,\bX_N)$, we invoke the bounded difference inequality in \cref{lem:bounded-difference} to deduce
	\begin{equation*}
		\begin{aligned}
			\bbP\Big\{\Big|\bZ-\bbE\bZ\Big|\geq\frac{\epsilon}{2}\Big\}\leq &~2\exp\Big(\!-C\frac{2\epsilon^2}{4\|f\|_{L_{\infty}}^{2}\sum_{j=1}^N\rho^{-d}N^{-2}}\Big)\\
			= &~2\exp\Big(\!-C\rho^{d}N\epsilon^2\|f\|_{L_{\infty}}^{-2}\Big),
		\end{aligned}
	\end{equation*}
	which completes the proof.
\end{proof}

From the proof of above theorem, we can derive the mean $L_2$-convergence and the maximal mean squared error (MMSE) convergence under random sampling.

\begin{corollary}(Mean $L_2$-convergence)\label{corol:L2Err-MC}
    Suppose that the scaled zonal kernel $\scaleKer$ in \eqref{eq:zonal_kernel} satisfies \cref{assump:kernel}. Let $f\in \Hsig(\bbS^d)$ with $\sigma\geq m>d/2$. Let MCQI $\cQ_{\scaleKer}^r f$ be defined in \eqref{eq:MCQI} with the parameter chosen as $\rho=\mathcal{O}(N^{-\frac{1}{2m+d}})$.  Then, there exists a constant $C>0$ independent of $N$ such that
    $$\bbE\Big[\|\cQ_{\scaleKer}^r f-f\|_{L_2}^2\Big]\leq CN^{-\frac{2m}{2m+d}}\|f\|_{\Hsig}.$$
\end{corollary}
\begin{proof}
    By combining \eqref{eq:ProbConv_2} and \eqref{eq:ProbConv_3}, and then selecting $\rho$ such that $\rho^m=\mathcal{O}(\rho^{-d/2}N^{-1/2})$, we get the desired result.
\end{proof}

\begin{corollary}(MMSE convergence)\label{corol:UniformErr-MC}
	Suppose that the scaled zonal kernel $\scaleKer$ in \eqref{eq:zonal_kernel} satisfies \cref{assump:kernel}. Let $f\in \Hsig(\bbS^d)$ with $\sigma\geq m>d/2$. Let MCQI $\cQ_{\scaleKer}^rf$ be defined in \eqref{eq:MCQI}. Then, there exists a constant $C>0$ independent of $\rho$ and $N$ such that
	$$\sup_{x\in\bbS^d}\bbE \big[\cQ_{\scaleKer}^r f(x)-f(x)\big]^2\leq C(\rho^{2m-d}+\rho^{-d}N^{-1})\|f\|_{\Hsig}^2.$$
    In particular, setting $\rho=\mathcal{O}(N^{-\frac{1}{2m}})$, we have
    $$\sup_{x\in\bbS^d}\bbE \big[\cQ_{\scaleKer}^r f(x)-f(x)\big]^2\leq C N^{-\frac{2m-d}{2m}}\|f\|_{\Hsig}^2.$$
\end{corollary}

\begin{proof}
	For a fixed point $x\in\bbS^d$, we can decompose the mean squared error as
	\begin{equation*}
		\begin{aligned}
			\bbE \big(\cQ_{\scaleKer}^r f(x)-f(x)\big)^2
			=\bbE\big[\cQ_{\scaleKer}^r f(x)-\mathcal{C}_{\scaleKer}f(x)\big]^2+\big(\mathcal{C}_{\scaleKer}f(x)-f(x)\big)^2.
		\end{aligned}
	\end{equation*}
	The second term on the right-hand side is bounded by \cref{thm:ConvErr},
	$$\big(\mathcal{C}_{\scaleKer}f(x)-f(x)\big)^2\leq C\rho^{2m-d}\|f\|_{\Hsig}^2.$$
	By using \eqref{eq:MSEerror1}, we can estimate the first term
\begin{equation*}
	\bbE\Big[ \big|   \cQ_{\scaleKer}^r f(x) -\mathcal{C}_{\scaleKer}f(x)\big|^2 \Big] \leq C\rho^{-d}N^{-1}\|f\|_{L^\infty}^2.
\end{equation*}
Stringing the two inequalities together, we arrive at
\begin{align*}
	\bbE\Big[ \big|\cQ_{\scaleKer}^r f(x) -f(x)\big|^2 \Big] 
    \leq C\big( \rho^{2m-d} + \rho^{-d}N^{-1} \big)\|f\|_{\Hsig}^2,
\end{align*}
where we have used $\|f\|_{L_{\infty}}\leq C\|f\|_{\Hsig}$ for $\sigma>d/2$. 
This estimate is uniform in $x\in\bbS^d$, which completes the proof.
\end{proof}

\begin{remark}
    The scaling parameter $\rho$ in the kernel governs a bias-variance trade-off in the approximation: decreasing $\rho$ reduces the deterministic convolution error ($\mathcal{O}(\rho^{2m-d})$), but increases the stochastic quasi-interpolation error ($\mathcal{O}(\rho^{-d}N^{-1})$). We can minimize the error with respect to $\rho$ that yields the optimal scaling $\rho=\mathcal{O}(N^{-\frac{1}{2m}})$. This choice then gives the uniform convergence rate $N^{-\frac{2m-d}{2m}}$.
\end{remark}

\section{Quasi-interpolation for noisy data on the sphere}
\label{sec:qi_noisydata}
In practical scenarios, data sampling inherently incorporates varying levels of noise. Consequently, developing methods capable of handling noise becomes essential.  Hesse et al. \cite{hesse-2017NM-radial} introduced a regularized interpolation-based approximation framework that balances data fidelity and kernel smoothness. Recently, the distributed learning literature has provided diverse strategies for addressing noisy data  \cite{lin_SINUM2021_distributed,montufar_2022FoCom_distributed}.
In contrast to kernel-based interpolation methods, the proposed SKQI approach does not strictly enforce interpolation conditions. This flexibility renders it particularly well-suited for approximating noisy functions, as it circumvents the need to solve regularization systems. 

Given a set of noisy sampling data on a quasi-Monte Carlo point set $X^q$, $$y_j=f(x_j)+\vare_j,~x_j\in X^q,~j=1,\ldots,N,$$
where $\{\vare_j\}_{j=1}^{N}$ is a set of independent random noises that satisfy $\bbE[\vare_j]=0$ and $|\vare_ j|\leq M_{\vare}$, we apply the QMC quasi-interpolation \eqref{eq:QMCQI} to obtain
\begin{equation}\label{eq:NoisyQI}
\cQ_{\scaleKer}^{q,\vare} f=\frac{1}{N}\sum_{j=1}^ {N}\big(f(x_j)+\vare_j\big)\scaleKer(x\cdot x_j),~~x\in \bbS^d.
\end{equation}

The following theorem shows an $L_2$-probabilistic concentration inequality of the QMC quasi-interpolation $\QiE f$ for approximating noisy data under deterministic sampling. 

\begin{theorem}($L_2$-probabilistic concentration inequality for noisy data under deterministic sampling)
\label{thm:L2ProbConv_noise}
    Let $\{\vare_j\}_{j=1}^{N}$ be a set of independent zero-mean random noises satisfying $|\vare_j|\leq M_{\vare}$ with $M_{\vare}>0$. Suppose $f\in H^{\sigma}(\bbS^d)$ with $\sigma\geq m>d/2$. Let the quasi-interpolation $\QiE f$ be defined in \eqref{eq:NoisyQI} with the scaled zonal kernel $\scaleKer$ satisfying \cref{assump:kernel} and \cref{assump2} with $s\geq \sigma$. 
    For any $\epsilon>0$, let $\rho$ and $N$ be chosen such that
    \begin{equation}\label{eq:eps_assump2}
        C_4\big({\rho}^{m}+\rho^{-d/2}N^{-1/2}\big)\big(\|f\|_{H^{\sigma}}+M_{\vare}\big)\leq \frac{\epsilon}{2},
    \end{equation}
    where $C_4$ is a constant given in \eqref{eq:ProbConv_noise_4}.
    Then, there exists a constant $C>0$ independent of $\rho$ and $N$ such that
	$$\bbP\Big\{\|\QiE f-f\|_{L_2}\geq\epsilon\Big\}\leq 2\exp\Big(\!-C\rho^d N\epsilon^2M_{\vare}^{-2}\Big).$$
\end{theorem}
\begin{proof}
We define the random variable $\bZ_{q}:=\|\QiE f-\cQ_{\scaleKer}^q f\|_2$ and apply the triangle inequality to write
\begin{equation}\label{eq:ProbConv_noise_1}
    \|\cQ_{\scaleKer}^{q,\vare} f-f\|_{L_2}\leq  |\bZ_{q}-\bbE \bZ_{q}| +\bbE\bZ_{q}+\|\cQ_{\scaleKer}^q f-f\|_{L_2}.
\end{equation}
The third term on the right-hand side can be bounded using \cref{thm:QMC_error} with $\tau=0$, 
\begin{equation*}\label{eq:ProbConv_noise_2}
    \|\cQ_{\scaleKer}^q f-f\|_{L_2}\leq C\big({\rho}^{m}+\rho^{-\sigma}N^{-\sigma/d}\big)\|f\|_{H^{\sigma}}.
\end{equation*}

To estimate the second term, we apply Jensen's inequality to obtain
\begin{equation*}
\begin{aligned}
    \bbE\bZ_{q}\leq \sqrt{ \bbE\bZ_{q}^2}
    = \Big(\int_{\bbS^d}\bbE\Big[\big(\QiE f(x)-\cQ_{\scaleKer}^q f(x)\big)^2\Big]\d\mu(x)\Big)^{1/2}.
\end{aligned}
\end{equation*}
Furthermore, by utilizing the independence and boundedness of $\vare_j$, we can write
\begin{equation*}
    \begin{aligned}
        \bbE\Big[\big(\QiE f(x)-\cQ_{\scaleKer}^q f(x)\big)^2\Big]
        = \bbE\Big[\Big(\frac{1}{N}\sum_{j=1}^ {N}\vare_j\scaleKer(x\cdot x_k)\Big)^2\Big]
        \leq  \frac{M_{\vare}^2}{N^2}\sum_{j=1}^ {N}|\scaleKer(x\cdot x_j)|^2,
    \end{aligned}
\end{equation*}
From which, we derive
\begin{equation*}
\begin{aligned}
    \bbE\bZ_{q}\leq  \Big(\frac{M_{\vare}^2}{N^2}\sum_{j=1}^ {N}\int_{\bbS^d}|\scaleKer(x\cdot x_j)|^2\d \mu(x)\Big)^{1/2}\leq C\rho^{-d/2}N^{-1/2}M_{\vare},
\end{aligned}
\end{equation*}
where the last inequality follows directly from \cref{lem:ScaleKernelNorm}.

Combining the results above, the inequality \eqref{eq:ProbConv_noise_1} can be further bounded by 
\begin{equation}\label{eq:ProbConv_noise_3}
    \|\QiE f-f\|_{L_2}\leq  |\bZ_{q}-\bbE \bZ_{q}|+C\big({\rho}^{m}+\rho^{-\sigma}N^{-\sigma/d}\big)\|f\|_{H^{\sigma}}+C\rho^{-d/2}N^{-1/2}M_{\vare}.
\end{equation}
Observing that $\rho^{-\sigma}N^{-\sigma/d}<\rho^{-d/2}N^{-1/2}$ for $\sigma>d/2$ and $N$ sufficiently large, \eqref{eq:ProbConv_noise_3} simplifies to
\begin{equation}\label{eq:ProbConv_noise_4}
    \|\QiE f-f\|_{L_2}\leq |\bZ_{q}-\bbE \bZ_{q}|+ C_4\big({\rho}^{m}+\rho^{-d/2}N^{-1/2}\big)\big(\|f\|_{H^{\sigma}}+M_{\vare}\big).
\end{equation}

For any $\epsilon>0$, using the assumption \eqref{eq:eps_assump2}, the following inequality holds true:
\begin{equation*}
		\bbP\Big\{\|\QiE f-f\|_{L_2}\geq\epsilon\Big\}\leq \bbP\Big\{|\bZ_q-\bbE\bZ_q|\geq\frac{\epsilon}{2}\Big\}.
\end{equation*}
Next, we introduce an auxiliary function:
\begin{align*}
    \cF_{q}(\varepsilon_1,\ldots,\varepsilon_N):=&~\big\|\frac{1}{N}\sum_{j=1}^N (f(x_j)+\varepsilon_j)\scaleKer(x\cdot x_j)-\frac{1}{N}\sum_{j=1}^N f(x_j)\scaleKer(x\cdot x_j)\big\|_{L_2}\\
    =&~\big\|\frac{1}{N}\sum_{j=1}^N \varepsilon_j\scaleKer(x\cdot x_j)\big\|_{L_2},~\text{with}~|\varepsilon_j|\leq M_{\vare},j=1,\ldots,N,
\end{align*}
and derive the following difference inequality
\begin{align*}
    &~|\cF_{q}(\varepsilon_1,\ldots,\varepsilon_j,\ldots,\varepsilon_N)-\cF_{q}(\varepsilon_1,\ldots,\varepsilon_j',\ldots,\varepsilon_N)|\\
    \leq &~\frac{1}{N}(|\varepsilon_j|+|\varepsilon_j'|)\|\scaleKer\|_{L_2}\leq C\rho^{-d/2}N^{-1}M_{\vare}.
\end{align*}
Note that $\bZ_q = \cF_{q}(\vare_1,\ldots,\vare_N)$. An application of \cref{lem:bounded-difference} then yields
	\begin{equation*}
		\begin{aligned}
			\bbP\Big\{\Big|\bZ_q-\bbE\bZ_q\Big|\geq\frac{\epsilon}{2}\Big\}\leq 2\exp\Big(\!-C\frac{2\epsilon^2}{4M_{\vare}^2\sum_{j=1}^N\rho^{-d}N^{-2}}\Big)
			= 2\exp\Big(\!-C\rho^{d}N\epsilon^2M_{\vare}^{-2}\Big),
		\end{aligned}
	\end{equation*}
    which completes the proof.
\end{proof}

From the above theorem, if we select the scaling parameter as $\rho=\mathcal{O}(N^{-\frac{1}{2m+d}})$, then an analogous mean $L_2$-convergence property can be derived. Remarkably, the proposed quasi-interpolation method preserves the same approximation order for noisy scattered data as established in \cref{corol:L2Err-MC} for clean data.

\begin{corollary}(Mean $L_2$-convergence for noisy data under deterministic sampling)
Let $\{\vare_j\}_{j=1}^{N}$ be a set of independent zero-mean random noises satisfying $|\vare_j|\leq M_{\vare}$ with $M_{\vare}>0$. Let $\cQ_{\scaleKer}^{q,\vare}$ be defined in \eqref{eq:MCQI}. Then, under the Assumptions of \cref{corol:L2Err-MC}, by choosing the scale parameter $\rho=\mathcal{O}(N^{\frac{1}{2m+d}})$, there exists a constant $C>0$ independent of $N$ such that  
\begin{equation*}
    \bbE\Big[\|f-\cQ_{\scaleKer}^{q,\vare} f\|_{L_2}^2\Big]\leq C N^{-\frac{2m}{2m+d}}\big(\|f\|_{H^{\sigma}}^2+M_{\vare}^2\big).
\end{equation*}
\end{corollary}

Furthermore, in the context of random sampling, we can formulate the following MCQI for noisy data:
\begin{equation}\label{eq:NoisyQI-random}
\cQ_{\scaleKer}^{r,\vare} f=\frac{1}{N}\sum_{j=1}^ {N}\big(f(\bX_j)+\vare_j\big)\scaleKer(x\cdot \bX_j),~~x\in \bbS^d,
\end{equation}
where $\vare_j$ are independent random noise variables with zero-mean and bounded variance $\sigma_{\vare}^2$.
Under this framework, we can establish an $L_2$-probabilistic concentration inequality analogous to that of the deterministic sampling case.
\begin{theorem}($L_2$-probabilistic concentration inequality for noisy data under random sampling)
\label{thm:L2ProbConv_noise_rand}
    Let $\{\vare_j\}_{j=1}^{N}$ be a set of independent random noises with zero-mean and bounded variance $\sigma_{\vare}^2$ satisfying $|\vare_j|\leq M_{\vare}$ with $M_{\vare}>0$. Let $f\in H^{\sigma}(\bbS^d)$ with $\sigma\geq m>d/2$. Let $\cQ_{\scaleKer}^{r,\vare} f$ be defined in \eqref{eq:NoisyQI-random} with the scaled zonal kernel $\scaleKer$ satisfying \cref{assump:kernel}. 
    For any $\epsilon>0$, let $\rho$ and $N$ be chosen such that
    \begin{equation}\label{eq:eps_assump3}
        C_5\Big({\rho}^{m}\|f\|_{H^{\sigma}}+\rho^{-d/2}N^{-1/2}\big(\|f\|_{L_{\infty}}+\sigma_{\vare}\big)\Big)\leq \frac{\epsilon}{2},
    \end{equation}
    where $C_5$ is a constant given in \eqref{eq:L2ProbConv_noise_rand}.
    Then, there exists a constant $C>0$ independent of $\rho$ and $N$ such that
	$$\bbP\Big\{\|\QiE f-f\|_{L_2}\geq\epsilon\Big\}\leq 2\exp\Big(\!-C\rho^{d}N\epsilon^2\big(\|f\|_{L_{\infty}}+M_{\vare}\big)^{-2}\Big).$$
\end{theorem}
\begin{proof}
    First, since $X_j$ and $\vare_j$ are independent random variables, we can compute the expectation of $\cQ_{\scaleKer}^{r,\vare} f$ as follows:
    \begin{align*}
        \bbE\cQ_{\scaleKer}^{r,\vare} f=&~\bbE\Big[\frac{1}{N}\sum_{j=1}^ {N}\big(f(\bX_j)+\vare_j\big)\scaleKer(x\cdot \bX_j)\Big]\\
        =&~\bbE\Big[\frac{1}{N}\sum_{j=1}^ {N}f(\bX_j)\scaleKer(x\cdot \bX_j)\Big]+\frac{1}{N}\sum_{j=1}^N\bbE[\vare_j]\bbE[\scaleKer(x\cdot \bX_j)]
        =\mathcal{C}_{\scaleKer}f,
    \end{align*}
    where the last inequality follows from the fact that $\bbE\vare_j=0$, $j=1,\ldots,N$.

    Next, we define the random variable $\bZ_r:=\|\cQ_{\scaleKer}^{r,\vare} f-\mathcal{C}_{\scaleKer}f\|_{L_2}$ and use the triangle inequality to write
    \begin{equation*}
        \|\cQ_{\scaleKer}^{r,\vare} f-f\|_{L_2}\leq |\bZ_r-\bbE\bZ_r|+\bbE\bZ_r+\|\mathcal{C}_{\scaleKer}f-f\|_{L_2}.
    \end{equation*}

   Since the remaining proof closely parallels the arguments in \cref{thm:ProbConv-MC} and \cref{thm:L2ProbConv_noise}, it suffices to estimate the second term and the third term explicitly. Once these two terms are estimated, the proof can be completed by following the same methodology as established therein. 
   
   The third term $\|\mathcal{C}_{\scaleKer}f-f\|_{L_2}$ is already bounded by \cref{lem:L2ConvErr} and inequality \eqref{eq:ProbConv_2}. To estimate the second term, we define the random variable
    $$\bY_{x,\vare}=\big(f(\bX_j)+\vare_j\big)\scaleKer(x\cdot \bX_j),$$
        and compute its second moment as follows:
\begin{equation*}
    \begin{aligned}
        \bbE [\bY_{x,\vare}^2]=&~\bbE f^2(\bX_j)\scaleKer^2(x\cdot\bX_j)+\bbE\vare_j^2\cdot\bbE\scaleKer^2(x\cdot\bX_j)\\
        =&~\int_{\bbS^d}f^2(y)\scaleKer^2(x\cdot y)\d\mu(y)+ \sigma_{\vare}^2\int_{\bbS^d}\scaleKer^2(x\cdot y)\d\mu(y)\\
        \leq &~C\rho^{-d}\big(\|f\|_{L_{\infty}}^2+\sigma_{\vare}^2\big),
    \end{aligned}
\end{equation*}
Using this result, we apply
    \begin{equation*}
        \bbE|\cQ_{\scaleKer}^{r,\vare} f(x)-\mathcal{C}_{\scaleKer}f(x)|^2 \leq N^{-1}\bbE [\bY_{x,\vare}^2],
    \end{equation*}
to obtain 
\begin{equation*}
    \begin{aligned}
    \bbE\bZ_r\leq \sqrt{\bbE\bZ_r^2}= &~\Big(\int_{\bbS^d}\bbE|\cQ_{\scaleKer}^{r,\vare} f(x)-\mathcal{C}_{\scaleKer}f(x)|^2\d\mu(x)\Big)^{1/2}\\
    \leq &~C\rho^{-d/2}N^{-1} \big(\|f\|_{L_{\infty}}+\sigma_{\vare}\big).
    \end{aligned}
\end{equation*}
Thus, we have
\begin{equation}\label{eq:L2ProbConv_noise_rand}
        \|\cQ_{\scaleKer}^{r,\vare} f-f\|_{L_2}\leq |\bZ_r-\bbE\bZ_r|+C_5\Big({\rho}^{m}\|f\|_{H^{\sigma}}+\rho^{-d/2}N^{-1/2}\big(\|f\|_{L_{\infty}}+\sigma_{\vare}\big)\Big).
    \end{equation}
This provides the desired estimate, which completes the proof.
\end{proof}

\section{Numerical examples}
\label{sec:numer_exper}
This section presents a suite of numerical experiments to evaluate the accuracy, stability and computational efficiency of SKQI methods for spherical data approximation.  
Our investigation proceeds through three distinct phases. First, we assess accuracy and convergence on various point sets using Gaussian kernels and compactly-supported kernels satisfying \cref{assump:kernel} with different orders $m$. Second, we analyze the convergence rates of single-level and multilevel quasi-interpolation schemes. Third, we conduct a comparative evaluation of SKQI against filtered hyperinterpolation (FHI) \cite{sloan_2012Geom_filtered} for the approximation of noisy data. The source code is available for reproducibility
at the repository \cite{sun2025githubcode}.
For accuracy tests, we adopt four point sets described in \cite{brauchart-2014MCoM-qmc}:
\begin{itemize}
    \item Pseudo-random points (RD), sampled uniformly on the unit sphere.    
    \item Maximal determinant points (MD), which maximize the determinant associated with interpolation.    
    \item Generalized spiral points (GS), with spherical coordinates $(\theta_{j},\phi_{j})$ given by
$$ z_{j}=1-\frac{2j-1}{N},\quad \theta_{j}=\cos^{-1}(z_{j}),\quad \phi_{j}=1.8\sqrt{N}\theta_{j}\bmod 2\pi,\quad j=1,\ldots,N.$$
 \item Symmetric spherical $t$-designs  (TD) \cite{sloan_2012Geom_filtered} with $N = 2(\lceil t(t+1)/4 \rceil+1)$ points.
\end{itemize}
 
  For random sampling, the maximum empirical mean square error over $J$ independent realizations is defined in \cite{gao-2020SISC-multivariate}:
  \[
     \text{MMSE} := 
        \max_{1 \leq k \leq M} \frac{1}{J} 
        \sum_{i=1}^{J} \left( \cQ_{\scaleKer} ^{r} f(x_k) - f(x_k) \right)^2.
  \]
Here, we set $J=100$ and $M=50000$. The evaluation point set $\{x_k\}_{k=1}^M$ is randomly distributed on the unit sphere.

\subsection{Convergence test}
We investigate the convergence properties of SKQI on various types of point sets distributed on the sphere. The method is tested using scaled Gaussian kernels and compactly-supported (CS) kernels satisfying \cref{assump:kernel} with $m=2,4,6$. The scaled parameter $\rho$ is determined according to our theoretical findings: $\rho = \mathcal{O}(N^{-\frac{1}{2m}})$ for RD in \cref{corol:UniformErr-MC} and $\rho = \mathcal{O}(N^{-\frac{1}{2d}})$ for QMC points in \cref{corol:QMCError}.

The spherical harmonic $\mathcal{Y}_{6,4}$ serves as the target function throughout this experiment.
The results, presented in \cref{fig:combined_L2_errors} and \cref{fig:combined_frank_L2_errors}, report approximation errors and their corresponding convergence rates of SKQI. For RD, the errors are quantified via MMSE, whereas for QMC points, $L_2$ errors are utilized.
The numerical results indicate that SKQI exhibits robust convergence across all point sets. For RD, the observed convergence rate closely conforms to the theoretical rate of $N^{-\frac{2m-d}{2m}}$ established in \cref{corol:UniformErr-MC}.  For QMC points, the convergence rate escalates with increasing order $m$, spanning from approximately $N^{-\frac{1}{2}}$ for $m=2$ to $N^{-\frac{3}{2}}$ for $m=6$, which is consistent with the theoretical bound $\mathcal{O}(N^{-\frac{m}{4}})$ derived in \cref{corol:QMCError} under the configuration $\tau=0$ and $\sigma=m$. Among QMC point sets, symmetric spherical $t$-design (TD) points yield the most favorable approximation errors and optimal convergence rates, presumably attributable to the superior uniformity inherent in this nodal distribution.

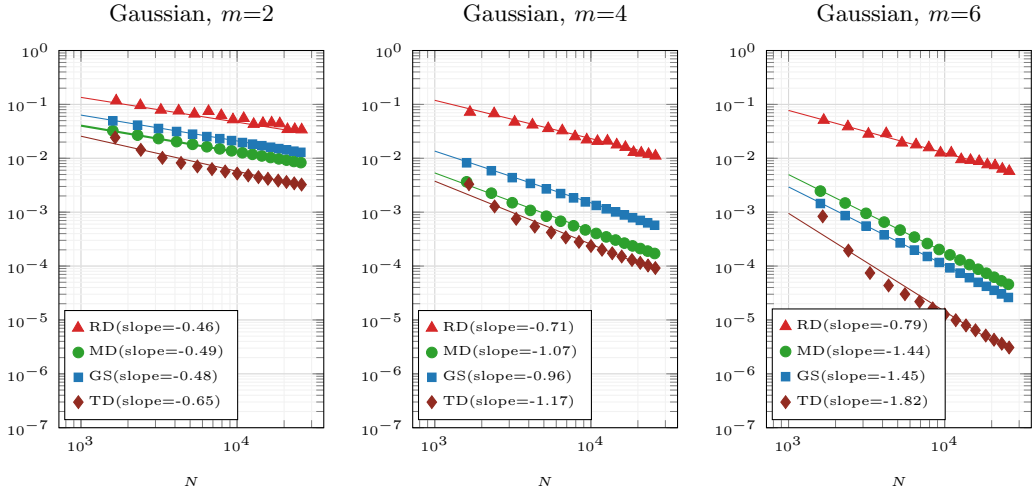
\begin{figure}
    \centering
    \begin{minipage}[t]{0.3\textwidth}
        \centering
        \begin{tikzpicture}
        \begin{loglogaxis}[
            width=1.15\linewidth,
            height=6.6cm,
            xlabel={$N$},
            title={Gaussian, $m$=2},
            legend style={
                at={(0.02,0.02)},
                anchor=south west,
                font=\tiny,
                row sep=0.1pt,
                legend columns=1,
                draw=black,
                fill=white
            },
            legend cell align=left,
            grid=both,
            major grid style={gray!30},
            minor grid style={gray!10},
            tick label style={font=\tiny},
            label style={font=\tiny},
            title style={font=\small},
            log basis x=10,
            log basis y=10,
            ymin=1e-7, ymax=1
        ]
        \addplot+[r1, only marks, mark=triangle*, mark options={solid,scale=1.2}]
            table[x=N, y=L2err0, col sep=tab] {RD_2.dat};
        \addlegendentry{RD(slope=-0.46)}
        \addplot[r1, domain=1e3:2e4, thin, forget plot] {3.235 * x ^ (-0.46)};

         \addplot+[g1, only marks, mark=*, mark options={solid,scale=1}]
            table[x=n, y=L2err0_M, col sep=tab] {together_2.dat};
        \addlegendentry{MD(slope=-0.49)}
        \addplot[g1, domain=1e3:2e4, thick, forget plot] {1.19 * x ^ (-0.49)};

        \addplot+[b1, only marks, mark=square*, mark options={solid,scale=0.8}]
            table[x=n, y=L2err0_G, col sep=tab] {together_2.dat};
        \addlegendentry{GS(slope=-0.48)}
        \addplot[b1, domain=1e3:2e4, thin, forget plot] {1.744 * x ^ (-0.48)};

        \addplot+[colorr, only marks, mark=diamond*, mark options={solid,scale=1.2}]
            table[x=Nx, y=L2err_T, col sep=tab] {together_2_T.dat};
        \addlegendentry{TD(slope=-0.65)}
        \addplot[colorr, domain=1e3:2e4, thin, forget plot] {2.289 * x ^ (-0.65)};
        \end{loglogaxis}
        \end{tikzpicture}
    \end{minipage}
    \hspace{0.01\textwidth}
    \begin{minipage}[t]{0.3\textwidth}
        \centering
        \begin{tikzpicture}
        \begin{loglogaxis}[
            width=1.15\linewidth,
            height=6.6cm,
            xlabel={$N$},
            title={Gaussian, $m$=4},
            legend style={
                at={(0.02,0.02)},
                anchor=south west,
                font=\tiny,
                row sep=0.1pt,
                legend columns=1,
                draw=black,
                fill=white
            },
            legend cell align=left,
            grid=both,
            major grid style={gray!30},
            minor grid style={gray!10},
            tick label style={font=\tiny},
            label style={font=\tiny},
            title style={font=\small},
            log basis x=10,
            log basis y=10,
            ymin=1e-7, ymax=1
        ]
      \addplot+[r1, only marks, mark=triangle*, mark options={solid,scale=1.2}]
            table[x=N, y=L2err0, col sep=tab] {RD_4.dat};
        \addlegendentry{RD(slope=-0.71)}
        \addplot[r1, domain=1e3:2e4, thin, forget plot] {16.06 * x ^ (-0.71)};

         \addplot+[g1, only marks, mark=*, mark options={solid,scale=1.0}]
            table[x=n, y=L2err0_M, col sep=tab] {together_4.dat};
        \addlegendentry{MD(slope=-1.07)}
        \addplot[g1, domain=1e3:2e4, thin, forget plot] {8.674 * x ^ (-1.07)};

        \addplot+[b1, only marks, mark=square*, mark options={solid,scale=0.8}]
            table[x=n, y=L2err0_G, col sep=tab] {together_4.dat};
        \addlegendentry{GS(slope=-0.96)}
        \addplot[b1, domain=1e3:2e4, thin, forget plot] {10.31 * x ^ (-0.96)};

        \addplot+[colorr, only marks, mark=diamond*, mark options={solid,scale=1.2}]
            table[x=Nx, y=L2err_T, col sep=tab] {together_4_T.dat};
        \addlegendentry{TD(slope=-1.17)}
        \addplot[colorr, domain=1e3:2e4, thin, forget plot] {12.17 * x ^ (-1.17)};
        \end{loglogaxis}
        \end{tikzpicture}
    \end{minipage}
    \hspace{0.01\textwidth}
    \begin{minipage}[t]{0.3\textwidth}
        \centering
        \begin{tikzpicture}
        \begin{loglogaxis}[
            width=1.15\linewidth,
            height=6.6cm,
            xlabel={$N$},
            title={Gaussian, $m$=6},
            legend style={
                at={(0.02,0.02)},
                anchor=south west,
                font=\tiny,
                row sep=0.3pt,
                legend columns=1,
                draw=black,
                fill=white,
            },
            legend cell align=left,
            grid=both,
            major grid style={gray!30},
            minor grid style={gray!10},
            tick label style={font=\tiny},
            label style={font=\tiny},
            title style={font=\small},
            log basis x=10,
            log basis y=10,
            ymin=1e-7, ymax=1
        ]
        \addplot+[r1, only marks, mark=triangle*, mark options={solid,scale=1.2}]
            table[x=N, y=L2err0, col sep=tab] {RD_6.dat};
        \addlegendentry{RD(slope=-0.79)}
        \addplot[r1, domain=1e3:2e4, thin, forget plot] {18.13 * x ^ (-0.79)};

        \addplot+[g1, only marks, mark=*, mark options={solid,scale=1}]
            table[x=n, y=L2err0_M, col sep=tab] {together_6.dat};
        \addlegendentry{MD(slope=-1.44)}
        \addplot[g1, domain=1e3:2e4,thin, forget plot] {104.1 * x ^ (-1.44)};

        \addplot+[b1, only marks, mark=square*, mark options={solid,scale=0.8}]
            table[x=n, y=L2err0_G, col sep=tab] {together_6.dat};
        \addlegendentry{GS(slope=-1.45)}
        \addplot[b1, domain=1e3:2e4, thin, forget plot] {65.83 * x ^ (-1.45)};

        \addplot+[colorr, only marks, mark=diamond*, mark options={solid,scale=1.2}]
            table[x=Nx, y=L2err_T, col sep=tab] {together_6_T.dat};
        \addlegendentry{TD(slope=-1.82)}
        \addplot[colorr, domain=1e3:2e4, thin, forget plot] {274.3 * x ^ (-1.82)};
        \end{loglogaxis}
        \end{tikzpicture}
    \end{minipage}
 \vspace{-0.2cm}
   \caption{Numerical errors and convergence rates of SKQI using \textbf{Gaussian kernels} with orders $m=2,4,6$ for approximating spherical harmonic $\mathcal{Y}_{6,4}$ on RD, MD, GS and TD point sets.}

    \label{fig:combined_L2_errors}
\end{figure} 

\begin{figure}
    \centering
    \begin{minipage}[t]{0.3\textwidth}
        \centering
        \begin{tikzpicture}
        \begin{loglogaxis}[
            width=1.15\linewidth,
            height=6.6cm,
            xlabel={$N$},
            title={CS, $m$=2},
            legend style={
                at={(0.02,0.02)},
                anchor=south west,
                font=\tiny,
                row sep=0.1pt,
                legend columns=1,
                draw=black,
                fill=white
            },
            legend cell align=left,
            grid=both,
            major grid style={gray!30},
            minor grid style={gray!10},
            tick label style={font=\tiny},
            label style={font=\tiny},
            title style={font=\small},
            log basis x=10,
            log basis y=10,
            ymin=1e-7, ymax=1
        ]
        \addplot+[r1, only marks, mark=triangle*, mark options={solid,scale=1.2}]
            table[x=N, y=L2err0, col sep=tab] {CS_2.dat};
        \addlegendentry{RD(slope=-0.56)}
        \addplot[r1, domain=1e3:2e4, thin, forget plot] {13.46 * x ^ (-0.56)};

         \addplot+[g1, only marks, mark=*, mark options={solid,scale=1}]
            table[x=N, y=L2err0_M, col sep=tab] {CS_2.dat};
        \addlegendentry{MD(slope=-0.50)}
        \addplot[g1, domain=1e3:2e4, thick, forget plot] {1.01 * x ^ (-0.50)};

        \addplot+[b1, only marks, mark=square*, mark options={solid,scale=0.8}]
            table[x=N, y=L2err0_G, col sep=tab] {CS_2.dat};
        \addlegendentry{GS(slope=-0.49)}
        \addplot[b1, domain=1e3:2e4, thin, forget plot] {1.661 * x ^ (-0.49)};

        \addplot+[colorr, only marks, mark=diamond*, mark options={solid,scale=1.2}]
            table[x=Nx, y=L2err0_T, col sep=tab] {CS_2.dat};
        \addlegendentry{TD(slope=-0.50)}
        \addplot[colorr, domain=1e3:2e4, thin, forget plot] {0.7187 * x ^ (-0.50)};
        \end{loglogaxis}
        \end{tikzpicture}
    \end{minipage}
    \hspace{0.01\textwidth}
    \begin{minipage}[t]{0.3\textwidth}
        \centering
        \begin{tikzpicture}
        \begin{loglogaxis}[
            width=1.15\linewidth,
            height=6.6cm,
            xlabel={$N$},
            title={CS, $m$=4},
            legend style={
                at={(0.02,0.02)},
                anchor=south west,
                font=\tiny,
                row sep=0.1pt,
                legend columns=1,
                draw=black,
                fill=white
            },
            legend cell align=left,
            grid=both,
            major grid style={gray!30},
            minor grid style={gray!10},
            tick label style={font=\tiny},
            label style={font=\tiny},
            title style={font=\small},
            log basis x=10,
            log basis y=10,
            ymin=1e-7, ymax=1
        ]
      \addplot+[r1, only marks, mark=triangle*, mark options={solid,scale=1.2}]
            table[x=N, y=L2err0, col sep=tab] {CS_4.dat};
        \addlegendentry{RD(slope=-0.61)}
        \addplot[r1, domain=1e3:2e4, thin, forget plot] {15.18 * x ^ (-0.61)};

         \addplot+[g1, only marks, mark=*, mark options={solid,scale=1.0}]
            table[x=N, y=L2err0_M, col sep=tab] {CS_4.dat};
        \addlegendentry{MD(slope=-1.00)}
        \addplot[g1, domain=1e3:2e4, thin, forget plot] {5.504 * x ^ (-1.00)};

        \addplot+[b1, only marks, mark=square*, mark options={solid,scale=0.8}]
            table[x=N, y=L2err0_G, col sep=tab] {CS_4.dat};
        \addlegendentry{GS(slope=-0.99)}
        \addplot[b1, domain=1e3:2e4, thin, forget plot] {9.438 * x ^ (-0.99)};

        \addplot+[colorr, only marks, mark=diamond*, mark options={solid,scale=1.2}]
            table[x=Nx, y=L2err0_T, col sep=tab] {CS_4.dat};
        \addlegendentry{TD(slope=-0.99)}
        \addplot[colorr, domain=1e3:2e4, thin, forget plot] {3.263 * x ^ (-0.99)};
        \end{loglogaxis}
        \end{tikzpicture}
    \end{minipage}
    \hspace{0.01\textwidth}
    \begin{minipage}[t]{0.3\textwidth}
        \centering
        \begin{tikzpicture}
        \begin{loglogaxis}[
            width=1.15\linewidth,
            height=6.6cm,
            xlabel={$N$},
            title={CS, $m$=6},
            legend style={
                at={(0.02,0.02)},
                anchor=south west,
                font=\tiny,
                row sep=0.3pt,
                legend columns=1,
                draw=black,
                fill=white,
            },
            legend cell align=left,
            grid=both,
            major grid style={gray!30},
            minor grid style={gray!10},
            tick label style={font=\tiny},
            label style={font=\tiny},
            title style={font=\small},
            log basis x=10,
            log basis y=10,
            ymin=1e-7, ymax=1
        ]
        \addplot+[r1, only marks, mark=triangle*, mark options={solid,scale=1.2}]
            table[x=N, y=L2err0, col sep=tab] {CS_6.dat};
        \addlegendentry{RD(slope=-0.65)}
        \addplot[r1, domain=1e3:2e4, thin, forget plot] {10.34 * x ^ (-0.65)};

        \addplot+[g1, only marks, mark=*, mark options={solid,scale=1}]
            table[x=N, y=L2err0_M, col sep=tab] {CS_6.dat};
        \addlegendentry{MD(slope=-1.49)}
        \addplot[g1, domain=1e3:2e4,thin, forget plot] {45.2 * x ^ (-1.49)};

        \addplot+[b1, only marks, mark=square*, mark options={solid,scale=0.8}]
            table[x=N, y=L2err0_G, col sep=tab] {CS_6.dat};
        \addlegendentry{GS(slope=-1.48)}
        \addplot[b1, domain=1e3:2e4, thin, forget plot] {82.64 * x ^ (-1.48)};

        \addplot+[colorr, only marks, mark=diamond*, mark options={solid,scale=1.2}]
            table[x=Nx, y=L2err0_T, col sep=tab] {CS_6.dat};
        \addlegendentry{TD(slope=-1.57)}
        \addplot[colorr, domain=1e3:2e4, thin, forget plot] {25.36 * x ^ (-1.57)};
        \end{loglogaxis}
        \end{tikzpicture}
    \end{minipage}
 \vspace{-0.2cm}
   \caption{Numerical errors and convergence rates of SKQI using \textbf{compactly-supported kernels} with $m=2,4,6$ for approximating spherical harmonic $\mathcal{Y}_{6,4}$ on RD, MD, GS and TD point sets.}

    \label{fig:combined_frank_L2_errors}
\end{figure}
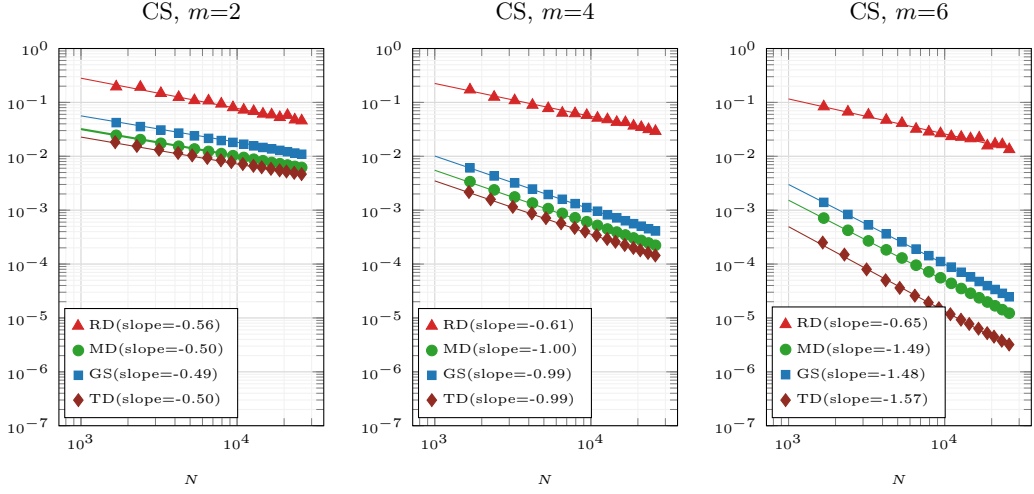

\subsection{Comparison of single-level and multilevel quasi-interpolation}
In this section, we conduct a comparative analysis of single-level and multilevel quasi-interpolation methods on the unit sphere. We employ MD point sets with cardinalities $N = 12^2, 24^2, 48^2, 96^2, 192^2$. The scaling parameter $\rho$ is selected as $\mathcal{O}(N^{-\frac{1}{4}})$.
\cref{fig.L2conv_randGaussian} illustrates the $L_\infty$ and $L_2$ errors of single-level and multilevel quasi-interpolation for clean data. The results reveal that the multilevel scheme substantially outperforms the single-level method in both accuracy and convergence, with higher kernel order $m$ exhibiting accelerated convergence rates.

To further demonstrate the generality of the proposed method, we also utilize the Franke function as a test case:
\begin{equation}\label{eq:frank}
\begin{aligned}
f(x, y, z) =\;& 0.75 \exp\big( -(9x - 2)^2/4- (9y - 2)^2/4-(9z - 2)^2/4 \big) \\
& + 0.75 \exp\big( -(9x + 1)^2/49-(9y + 1)/10 - (9z + 1)/10 \big) \\
& + 0.5 \exp\big( -(9x - 7)^2/4 - (9y - 3)^2/4- (9z - 5)^2/4 \big) \\
& - 0.2 \exp\big( -(9x - 4)^2 - (9y - 7)^2 - (9z - 5)^2 \big), \quad (x,y,z) \in \mathbb{S}^2.
\end{aligned}
\end{equation} 
\cref{tab:single-multi-results} tabulates the approximation errors for noisy data using both methods under noise levels $\sigma_{\varepsilon} = 0.01$ and $\sigma_{\varepsilon} = 0.1$. The findings establish that the multilevel scheme maintains lower errors and superior convergence rates compared to the single-level scheme, even in the presence of noise. Additionally, \cref{fig:error_comparison} depicts the $L_\infty$ and $L_2$ errors under noise levels $\sigma_{\varepsilon} = 0.001, 0.01, 0.1$, which corroborates that the multilevel scheme achieves smaller errors than the single-level scheme.

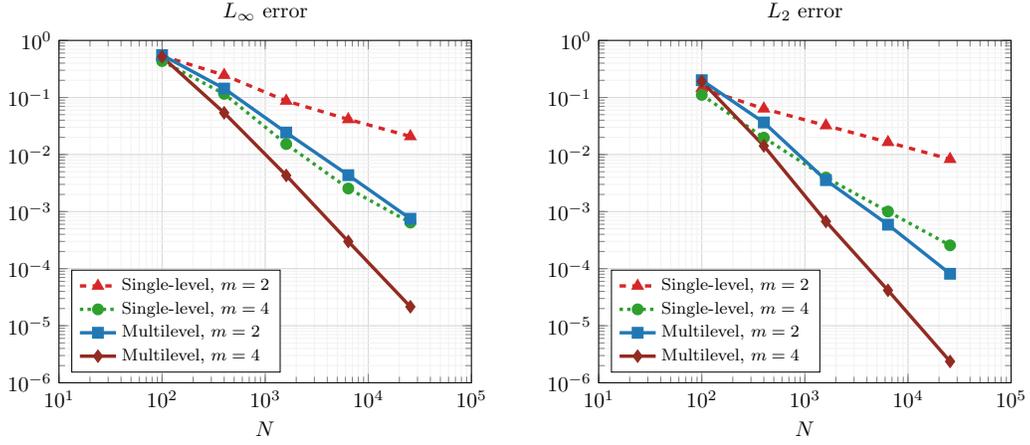
\begin{figure}
	\centering
 \begin{tikzpicture}[scale=0.8]
		\begin{loglogaxis}[
			    grid=both,
                major grid style={gray!30},
                minor grid style={gray!10},
			mark size = 1.5pt,
			xmin = 1e1, xmax = 1e5,
			ymin = 10^(-6),ymax = 10^(-0),
			xlabel={$N$},
			title={$L_{\infty}$ error},
			legend cell align = {left},
			legend pos = south west,
			legend style = {font=\footnotesize},
			legend entries={{Single-level, $m=2$},{Single-level, $m=4$}, {Multilevel, $m=2$},{Multilevel, $m=4$}}]
			\addplot  +[ultra thick,dashed, mark = triangle*,mark options={solid}, mark size = 2pt, color = r1] table [x=N,y=Linferr_m2] {Linferr_gauss_sphHarm64_single_N.dat};
			\addplot  +[ultra thick,  dotted, mark = *,mark options={solid},mark size = 2pt, color = g1] table [x=N,y=Linferr_m4] {Linferr_gauss_sphHarm64_single_N.dat};
            \addplot  +[ultra thick, mark =square*,mark options={solid}, mark size = 2pt, color = b1] table [x=N,y=Linferr_m2] {Linferr_gauss_sphHarm64_multi_N.dat};
			\addplot  +[ultra thick,  mark =diamond*, mark size = 2pt, color = colorr] table [x=N,y=Linferr_m4] {Linferr_gauss_sphHarm64_multi_N.dat};            
		\end{loglogaxis}
	\end{tikzpicture}
 	\hspace{0.5cm}
 \begin{tikzpicture}[scale=0.8]
		\begin{loglogaxis}[
			    grid=both,
                major grid style={gray!30},
                minor grid style={gray!10},
			xmin = 1e1, xmax = 1e5,
			ymin = 10^(-6),ymax = 10^(-0),
			xlabel={$N$},
			title={$L_2$ error},
			legend cell align = {left},
			legend pos = south west,
			legend style = {font=\footnotesize},
			legend entries={{Single-level, $m=2$},{Single-level, $m=4$}, {Multilevel, $m=2$},{Multilevel, $m=4$}}]
\addplot  +[ultra thick, dashed, mark = triangle*, mark options={solid}, mark size = 2pt, color = r1] 
    table [x=N,y=L2err_m2] {L2err_gauss_sphHarm64_single_N.dat};
    
\addplot  +[ultra thick, dotted, mark =*, mark options={solid}, mark size = 2pt, color =g1] 
    table [x=N,y=L2err_m4] {L2err_gauss_sphHarm64_single_N.dat};
    
\addplot  +[ultra thick, mark = square*, mark options={solid}, mark size = 2pt, color = b1] 
    table [x=N,y=L2err_m2] {L2err_gauss_sphHarm64_multi_N.dat};
    
\addplot  +[ultra thick, mark =diamond*, mark options={solid}, mark size = 2pt, color = colorr] 
    table [x=N,y=L2err_m4] {L2err_gauss_sphHarm64_multi_N.dat};

		\end{loglogaxis}
	\end{tikzpicture}
    \vspace{-0.2cm}
	\caption{$L_\infty$ and $L_2$ errors for approximating $\mathcal{Y}_{6,4}$ using single-level and multilevel SKQI methods with compactly-supported kernels ($m=2,4$) on MD point sets. The left figure corresponds to $L_{\infty}$ error, and the right corresponds to $L_2$ error.}
	\label{fig.L2conv_randGaussian}
\end{figure}

\begin{table}
\centering
\caption{Approximation errors of single-level and multilevel SKQI methods using the compactly-supported kernel with $m=2$ for approximating Franke function under two noise levels ($\sigma_{\vare}=0.01$ and $\sigma_{\vare}=0.1$) on MD point sets.}
 \vspace{-0.2cm}
\label{tab:single-multi-results}
\begin{tabular}{c|cc|cc}
\toprule
\multirow{2}{*}{$N$} & \multicolumn{2}{c|}{Noise level $\sigma_{\vare} = 0.01$} & \multicolumn{2}{c}{Noise level $\sigma_{\vare} = 0.1$} \\
& $L_{\infty}$ error & $L_{2}$ error & $L_{\infty}$ error & $L_{2}$ error \\
\midrule
\multicolumn{5}{c}{Single-level} \\
\midrule
$12^2$ & $7.70 \times 10^{-1}$ & $2.62 \times 10^{-1}$ & $9.70 \times 10^{-1}$ & $3.08 \times 10^{-1}$ \\
$24^2$ & $5.03 \times 10^{-1}$ & $1.65 \times 10^{-1}$ & $6.01 \times 10^{-1}$ & $1.94 \times 10^{-1}$ \\
$48^2$ & $2.97 \times 10^{-1}$ & $9.32 \times 10^{-2}$ & $4.38 \times 10^{-1}$ & $1.17 \times 10^{-1}$ \\
$96^2$ & $1.94 \times 10^{-1}$ & $4.95 \times 10^{-2}$ & $2.73 \times 10^{-1}$ & $7.18 \times 10^{-2}$ \\
$192^2$ & $1.12 \times 10^{-1}$ & $2.55 \times 10^{-2}$ & $1.72 \times 10^{-1}$ & $4.33 \times 10^{-2}$ \\
\midrule
\multicolumn{5}{c}{Multilevel} \\
\midrule
$12^2$ & $7.78 \times 10^{-1}$ & $2.98 \times 10^{-1}$ & $8.08 \times 10^{-1}$ & $3.00 \times 10^{-1}$ \\
$24^2$ & $2.59 \times 10^{-1}$ & $9.34 \times 10^{-2}$ & $2.88 \times 10^{-1}$ & $9.72 \times 10^{-2}$ \\
$48^2$ & $6.84 \times 10^{-2}$ & $1.73 \times 10^{-2}$ & $1.31 \times 10^{-1}$ & $3.17 \times 10^{-2}$ \\
$96^2$ & $4.15 \times 10^{-2}$ & $1.01 \times 10^{-2}$ & $8.23 \times 10^{-2}$ & $1.98 \times 10^{-2}$ \\
$192^2$ & $2.95 \times 10^{-2}$ & $6.82 \times 10^{-3}$ & $5.98 \times 10^{-2}$ & $1.36 \times 10^{-2}$ \\
\bottomrule
\end{tabular}
\end{table}

\begin{figure}
    \centering
\begin{minipage}[t]{0.48\textwidth}
    \centering
    \begin{tikzpicture}
        \begin{loglogaxis}[
            width=\linewidth, height=6.5cm,
            grid=both,
            xmin=100, xmax=70000,
            ymin=1e-4, ymax=1e0,
            xlabel={$N$},
            title={$L_{\infty}$-error},
            legend style={
                font=\tiny,
                at={(0.02,0.02)},
                anchor=south west,
                cells={anchor=west},
                draw=black,
                fill=white
            },
            legend cell align=left,
            tick label style={font=\footnotesize},
            label style={font=\footnotesize},
            title style={font=\footnotesize},
            major grid style={gray!30},
            minor grid style={gray!10},
            log basis x=10,
            log basis y=10
        ]
        \pgfplotstableread[col sep=space]{multi_single_frank.dat}\datatable
        \pgfplotstableread[col sep=space]{multi_single_frank0001.dat}\datatablenew
        \addplot+[red!80!black, thick, solid, mark=triangle*, mark size=2.5pt, mark options={solid}, forget plot]
            table[x=all_N, y=INF_Error_multi] {\datatablenew};
        \addplot+[blue!80!black, thick, solid, mark=square*, mark size=1.8pt, mark options={solid}, forget plot]
            table[x=all_N, y=INF_Error_multi(col2)] {\datatable};
        \addplot+[green!80!black, thick, solid, mark=*, mark size=2pt, mark options={solid}, forget plot]
            table[x=all_N, y=INF_Error_multi(col3)] {\datatable};
        \addplot+[colorr, line width=1pt, dashed, mark=diamond*, mark size=2pt, mark options={solid},forget plot]
            table[x=all_N, y=inf_err_sphQI(col1)] {\datatable};

        \addplot+[coloro, line width=1pt, dashed, mark=pentagon*, mark size=2pt, mark options={solid}, forget plot]
            table[x=all_N, y=inf_err_sphQI(col2)] {\datatable};

        \addplot+[colordgr, line width=1pt, dashed, mark=star, mark size=2pt, mark options={solid}, forget plot]
            table[x=all_N, y=inf_err_sphQI(col3)] {\datatable};
\node[anchor=south west, fill=white, draw=black, inner sep=3pt] at (rel axis cs:0.02,0.02) {
      \begin{tikzpicture}[font=\Large, baseline, scale=0.55, transform shape]
      \node[anchor=west,font=\Large] at (0.3,0) {Single-level:};
      \draw[colorr, thick, dash pattern=on 1.5pt off 1.2pt] (1.7,0) -- (2.1,0);
      \node[text=colorr, anchor=west, inner sep=2.5pt] at (1.8,0) {$\blacklozenge$};
      \node[anchor=west] at (2.1,0) {$0.001,$};
      
      \draw[coloro, thick, dash pattern=on 1.5pt off 1.2pt] (2.9,0) -- (3.3,0);
      \node[draw=coloro, fill=coloro, regular polygon, regular polygon sides=5,
      minimum size=4.5pt, inner sep=0pt, anchor=west, transform shape, scale=2.6] 
      at (3.02,0) {};      
      \node[anchor=west] at (3.3,0) {$0.01,$};
      
      \draw[colordgr, thick, dash pattern=on 1.5pt off 1.2pt] (4.1,0) -- (4.5,0);
      \node[text=colordgr, anchor=west] at (4.12,0){$\bigstar$};
      \node[anchor=west] at (4.5,0) {$0.1$};

     \node[ anchor=west] at (0.3,-0.7) {Multilevel:};
      \draw[r1, thick] (1.75,-0.7) -- (2.1,-0.7); 
      \node[text=r1, anchor=west, inner sep=2.5pt] at (1.8,-0.7){$\blacktriangle$};
      \node[anchor=west] at (2.1,-0.7) {$0.001,$};
      
      \draw[blue!80!black, thick] (2.9,-0.7) -- (3.3,-0.7); 
      \node[text=blue!80!black, anchor=west,inner sep=2.5pt] at (2.95,-0.7) {$\blacksquare$};
      \node[anchor=west] at (3.3,-0.7) {$0.01,$};
      
      \draw[g1, thick] (4.1,-0.7) -- (4.5,-0.7); 
      \node[text=g1, anchor=west, inner sep=2.5pt] at (4.15,-0.7) {\scalebox{1.5}{$\bullet$}};
      \node[anchor=west] at (4.5,-0.7) {$0.1$};

    \end{tikzpicture}
  };
        \end{loglogaxis}
    \end{tikzpicture}
\end{minipage}
\hfill
\begin{minipage}[t]{0.48\textwidth}
    \centering
    \begin{tikzpicture}
        \begin{loglogaxis}[
            width=\linewidth, height=6.5cm,
            grid=both,
            xmin=100, xmax=70000,
            ymin=1e-4, ymax=1e0,
            xlabel={$N$},
            title={$L_2$-error},
            legend style={
                font=\tiny,
                at={(0.02,0.02)},
                anchor=south west,
                cells={anchor=west}
            },
            legend cell align=left,
            tick label style={font=\footnotesize},
            label style={font=\footnotesize},
            title style={font=\footnotesize},
            major grid style={gray!30},
            minor grid style={gray!10},
            log basis x=10,
            log basis y=10
        ]

        \pgfplotstableread[col sep=space]{multi_single_frank.dat}\datatable
        \pgfplotstableread[col sep=space]{multi_single_frank0001.dat}\datatablenew

\addplot+[red!80!black, thick, solid, mark=triangle*, mark size=2.5pt, mark options={solid}, forget plot]
    table[x=all_N, y=L2err_multi] {\datatablenew};
\addplot+[blue!80!black, thick, solid, mark=square*, mark size=1.8pt, mark options={solid}, forget plot]
    table[x=all_N, y=L2err_multi(col2)] {\datatable};
\addplot+[green!80!black, thick, solid, mark=*, mark size=2pt, mark options={solid}, forget plot]
    table[x=all_N, y=L2err_multi(col3)] {\datatable};

\addplot+[colorr, line width=1pt, dashed, mark=diamond*, mark size=2pt, mark options={solid}, forget plot]
    table[x=all_N, y=L2err_sphQI(col1)] {\datatable};
\addplot+[coloro, line width=1pt, dashed, mark=pentagon*, mark size=2pt, mark options={solid}, forget plot]
    table[x=all_N, y=L2err_sphQI(col2)] {\datatable};
\addplot+[colordgr, line width=1pt, dashed, mark=star, mark size=2pt, mark options={solid}, forget plot]
    table[x=all_N, y=L2err_sphQI(col3)] {\datatable};

\node[anchor=south west, fill=white, draw=black, inner sep=3pt] at (rel axis cs:0.02,0.02) {
      \begin{tikzpicture}[font=\Large, baseline, scale=0.55, transform shape]
      \node[anchor=west,font=\Large] at (0.3,0) {Single-level:};
      \draw[colorr, thick, dash pattern=on 1.5pt off 1.2pt] (1.7,0) -- (2.1,0);
      \node[text=colorr, anchor=west, inner sep=2.5pt] at (1.8,0) {$\blacklozenge$};
      \node[anchor=west] at (2.1,0) {$0.001,$};
      
      \draw[coloro, thick, dash pattern=on 1.5pt off 1.2pt] (2.9,0) -- (3.3,0);
      \node[draw=coloro, fill=coloro, regular polygon, regular polygon sides=5,
      minimum size=4.5pt, inner sep=0pt, anchor=west, transform shape, scale=2.6] 
      at (3.02,0) {};      
      \node[anchor=west] at (3.3,0) {$0.01,$};
      
      \draw[colordgr, thick, dash pattern=on 1.5pt off 1.2pt] (4.1,0) -- (4.5,0);
      \node[text=colordgr, anchor=west] at (4.12,0){$\bigstar$};
      \node[anchor=west] at (4.5,0) {$0.1$};

     \node[ anchor=west] at (0.3,-0.7) {Multilevel:};
      \draw[r1, thick] (1.75,-0.7) -- (2.1,-0.7); 
      \node[text=r1, anchor=west, inner sep=2.5pt] at (1.8,-0.7){$\blacktriangle$};
      \node[anchor=west] at (2.1,-0.7) {$0.001,$};
      
      \draw[blue!80!black, thick] (2.9,-0.7) -- (3.3,-0.7); 
      \node[text=blue!80!black, anchor=west,inner sep=2.5pt] at (2.95,-0.7) {$\blacksquare$};
      \node[anchor=west] at (3.3,-0.7) {$0.01,$};
      
      \draw[g1, thick] (4.1,-0.7) -- (4.5,-0.7); 
      \node[text=g1, anchor=west, inner sep=2.5pt] at (4.15,-0.7) {\scalebox{1.5}{$\bullet$}};
      \node[anchor=west] at (4.5,-0.7) {$0.1$};     

    \end{tikzpicture}
  };

        \end{loglogaxis}
    \end{tikzpicture}
\end{minipage}
\vspace{-0.2cm}
\caption{$L_\infty$ and $L_2$ errors for approximating Franke function \eqref{eq:frank} under various noise levels ($\sigma_{\vare}=0.001, 0.01, 0.1$) using single-level and multilevel SKQI methods with the compactly-supported kernel ($m=2$) on MD point sets.}
\label{fig:error_comparison}
\end{figure}
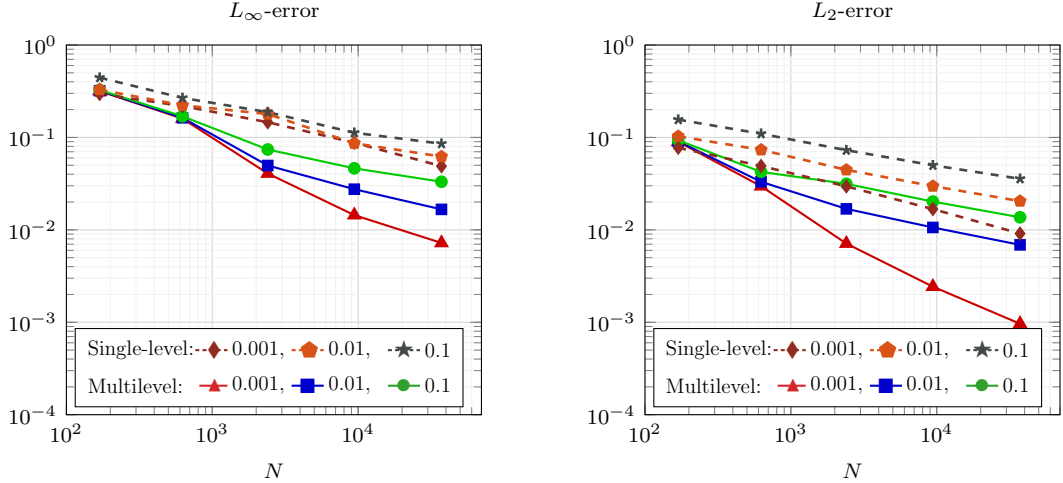

\subsection{Comparison with FHI}
The authors of \cite{sun-gao-sun} established that SKQI surpasses the hyperinterpolation (HI) method for noisy data with respect to both convergence properties and computational efficiency. Drawing upon these observations, in this section we evaluate the proposed SKQI methods, encompassing QMCQI on SD point sets and MCQI on random points, against the filtered hyperinterpolation approach formulated in
\cite{sloan_2012Geom_filtered}:
$$V_L^{(a)}f(x)=\sum_{\ell=0}^{\bar{L}(a)}h\Big(\frac{\ell}{L}\Big)\sum_{k=1}^{Z(d,\ell)}\Hatf\sphHarm(x),~\bar{L}(a):=\max(\lceil aL\rceil-1,L),$$
with the filtered kernel defined as \cite{filbir-2008MathNachr-polynomial, sloan_2012Geom_filtered}:
\[
h(x) = 
\begin{cases} 
1, & x \in [0, 1], \\[6pt]
\exp\Big(\dfrac{2\exp\left(-2 [y(x)]^{-1}\right)}{ y(x)-1}\Big), & x \in (1, a), \\[6pt]
0, & x \in [a, \infty),
\end{cases}
\]
where $y(x) = \dfrac{x - 1}{a - 1}$ and the parameter is chosen as $a = 1.2$.

\cref{fig:combined} displays the $L_2$ errors of FHI, QMCQI ($m=2,4$) on SD point sets and MCQI ($m=2,4$) under two noise levels: $\sigma_{\vare}=0.01$ and $\sigma_{\vare}=0.1$. It can be observed that for both QMCQI and MCQI, the $L_2$ approximation error decreases as $N$ increases, with higher-order kernels conferring enhanced accuracy. Remarkably, both SKQI methods sustain stable convergence even at the elevated noise level ($\sigma_{\vare}=0.1$), while filtered hyperinterpolation fails to converge in both cases. \cref{fig:computation_time_three} contrasts the computational costs among HI \cite{sloan_1995JAT_polynomial}, FHI and our QMCQI method. The result demonstrates that the proposed quasi-interpolation method delivers markedly superior computational performance.

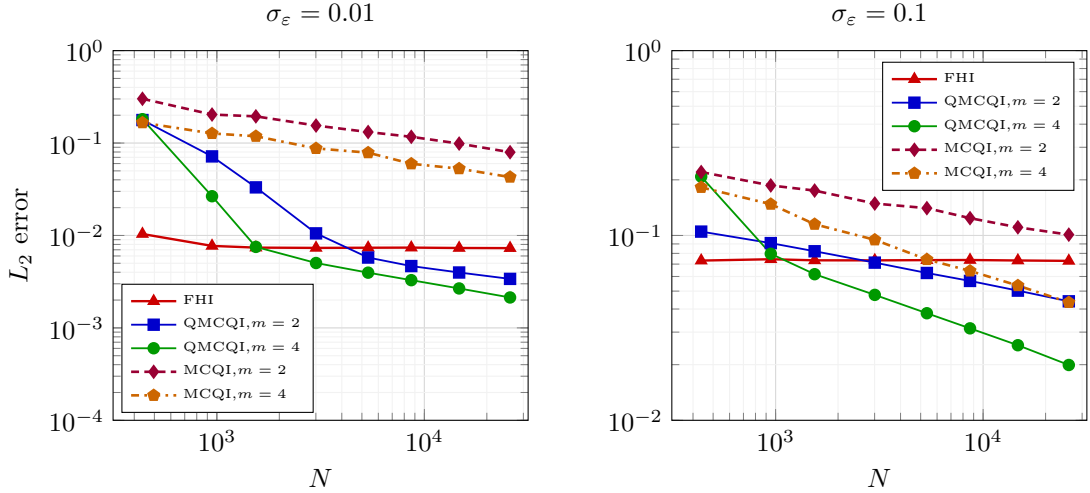
\begin{figure}
    \centering
    \begin{minipage}{0.48\textwidth}
        \centering
        \begin{tikzpicture}
            \begin{loglogaxis}[
                width=\linewidth, height=6.5cm,
                grid=both,
                major grid style={gray!30},
                minor grid style={gray!10},
                xmin= 316.2278, xmax=31622.78,
                ymin=1e-4, ymax=1e0,
                xlabel={$N$},
                ylabel={$L_2$ error},
                title={$\sigma_{\varepsilon}=0.01$},
                legend style={
                    font=\tiny,
                    at={(0.02,0.02)},
                    anchor=south west,
                    cells={anchor=west}
                }
            ]
                \pgfplotstableread[col sep=space]{filter_001_4_2.dat}\datatable
                \pgfplotstableread[col sep=space]{L2err_hyper_001.dat}\hyperdata
                \pgfplotstableread{qi_nosie_4_02.dat}\datatabletwo
                \pgfplotstableread[col sep=space]{qi_nosie_RD001_2.dat}\datatablethree
                \pgfplotstableread[col sep=space]{qi_nosie_RD001_4.dat}\datatablefour                             
                
                \addplot+[color=red!80!black, solid, line width=0.9pt, mark=triangle*, mark options={solid, fill=red!80!black}, mark size=2pt] 
                    table[x=Nx, y=L2err_hyper] {\hyperdata};
                \addlegendentry{FHI}
                
                \addplot+[color=blue!80!black, solid, line width=0.7pt, mark=square*, mark options={solid, fill=blue!80!black}, mark size=2pt]
                    table[x=Nx, y=L2err_sphQI_2] {\datatable};
                \addlegendentry{QMCQI,$m=2$}

                \addplot+[color=green!60!black, solid, line width=0.7pt, mark=*, mark options={solid, fill=green!60!black}, mark size=2pt] 
                    table[x=Nx, y=L2err_sphQI] {\datatabletwo};
                \addlegendentry{QMCQI,$m=4$}

                \addplot+[color=purple!80!black, line width=1pt, densely dashed, mark=diamond*, mark options={solid, fill=purple!80!black}, mark size=2pt]
                    table[x=Nx, y=L2err0] {\datatablethree};
                \addlegendentry{MCQI,$m=2$}

                \addplot+[color=orange!80!black, line width=1pt, dashdotted, mark=pentagon*, mark options={solid, fill=orange!80!black}, mark size=2pt]
                    table[x=Nx, y=L2err0] {\datatablefour};
                \addlegendentry{MCQI,$m=4$}                

            \end{loglogaxis}
        \end{tikzpicture}
    \end{minipage}
    \hfill 
    \begin{minipage}{0.48\textwidth}
        \centering
        \begin{tikzpicture}
            \begin{loglogaxis}[
                width=\linewidth, height=6.5cm,
                grid=both,
                major grid style={gray!30},
                minor grid style={gray!10},
                xmin= 316.2278, xmax=31622.78,
                ymin=1e-2, ymax=1e0,
                xlabel={$N$},
                title={$\sigma_{\varepsilon}=0.1$},
                legend cell align = {left},
                legend pos = north east,
                legend style = {font=\tiny},
                legend entries={
                    {FHI},
                    {QMCQI,$m=2$},
                    {QMCQI,$m=4$}, 
                    {MCQI,$m=2$},
                    {MCQI,$m=4$}
                }
            ]
            \pgfplotstableread[col sep=space]{qi_nosie_01_2.dat}\datatable
            \pgfplotstableread[col sep=space]{L2err_hyper_01.dat}\hyperdata
            \pgfplotstableread[col sep=space]{qi_nosie01_4_02.dat}\datatabletwo
            \pgfplotstableread[col sep=space]{qi_nosie_RD01_2.dat}\datatablethree
            \pgfplotstableread[col sep=space]{qi_nosie_RD01_4.dat}\datatablefour

                \addplot+[color=red!80!black, solid, line width=0.9pt, mark=triangle*, mark options={solid, fill=red!80!black}, mark size=2pt] 
                    table[x=Nx, y=L2err_hyper] {\hyperdata};
                    
                \addplot+[color=blue!80!black, solid, line width=0.7pt, mark=square*, mark options={solid, fill=blue!80!black}, mark size=2pt] 
                    table[x=Nx, y=L2err_sphQI] {\datatable};
                    
                \addplot+[color=green!60!black, solid, line width=0.7pt, mark=*, mark options={solid, fill=green!60!black}, mark size=2pt] 
                    table[x=Nx, y=L2err_sphQI] {\datatabletwo};
                    
                \addplot+[color=purple!80!black, line width=1pt, densely dashed, mark=diamond*, mark options={solid, fill=purple!80!black}, mark size=2pt]
                    table[x=Nx, y=L2err0] {\datatablethree};
                    
                \addplot+[color=orange!80!black, line width=1pt, dashdotted, mark=pentagon*, mark options={solid, fill=orange!80!black}, mark size=2pt]
                    table[x=Nx, y=L2err0] {\datatablefour};
                    
            \end{loglogaxis}
        \end{tikzpicture}
    \end{minipage}

    \vspace{-0.2cm}
    \caption{$L_2$ errors of FHI, QMCQI ($m=2,4$) and MCQI ($m=2,4$) for approximating Franke function \eqref{eq:frank} under two noise levels, $\sigma_{\varepsilon}=0.01$ and $\sigma_{\varepsilon}=0.1$.}
    \label{fig:combined}
\end{figure}

\begin{figure}
\centering
\begin{tikzpicture}[scale=0.8]
\begin{axis}[
    ybar=0pt,
    bar width=10pt,
    width=13cm,
    height=7cm,
    enlarge x limits=0.2,
    ylabel={Computation time (s)},
    xlabel={$N$},
    symbolic x coords={1328,2704,4562,9732,13044},
    xtick=data,
    ymin=0,
    ymax=20,
    ymajorgrids=true,
    grid style={dashed,gray!30},
    legend style={
        at={(0.5,-0.2)},
        anchor=north,
        legend columns=3,
        /tikz/every even column/.append style={column sep=0.5cm},
        font=\footnotesize,
        draw=none
    },
    nodes near coords,
    every node near coord/.append style={font=\tiny, yshift=2pt},
    nodes near coords align={vertical},
]

\addplot+[style={fill=b1}, bar shift=-12pt]
    coordinates {
       (1328,1.4) (2704,3.2) (4562,5.2) (9732,11.7) (13044,16.5)
    };
\addlegendentry{FHI}

\addplot+[style={fill=g1}, bar shift=0pt]
    coordinates {
       (1328,1.2) (2704,2.7) (4562,4.3) (9732,9.7) (13044,13.6)
    };
\addlegendentry{HI}

\addplot+[style={fill=r1}, bar shift=12pt] 
    coordinates {
       (1328,1.2) (2704,1.9) (4562,2.7) (9732,4.3) (13044,4.6)
    };
\addlegendentry{QMCQI}

\end{axis}
\end{tikzpicture}
 \vspace{-0.2cm}
\caption{
  Computational times of HI, FHI, and QMCQI for approximating Franke function \eqref{eq:frank} on TD point sets with $t=51,73,95,139,161$.
}
\label{fig:computation_time_three}
\end{figure}
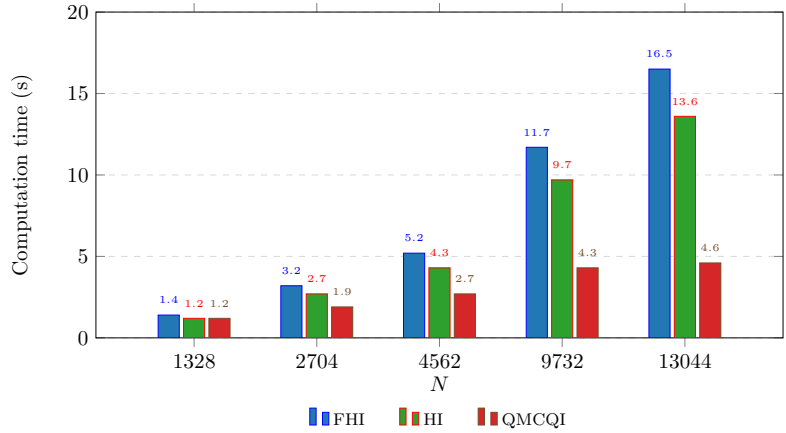



\bibliographystyle{plain}
\bibliography{ref}

\begin{thebibliography}{10}

\bibitem{an_SINUM2012_regularized}
C.~An, X.~Chen, I.H. Sloan, and R.S. Womersley.
\newblock Regularized least squares approximations on the sphere using
  spherical designs.
\newblock {\em SIAM J. Numer. Anal.}, 50(3):1513--1534, 2012.

\bibitem{brauchart-2014MCoM-qmc}
J.S. Brauchart, E.B. Saff, I.H. Sloan, and R.S. Womersley.
\newblock {QMC} designs: optimal order quasi {M}onte {C}arlo integration
  schemes on the sphere.
\newblock {\em Math. Comput.}, 83(290):2821--2851, 2014.

\bibitem{buhmann2022quasi-book}
M.~Buhmann and J.~J{\"a}ger.
\newblock {\em Quasi-interpolation}, volume~37.
\newblock Cambridge University Press, 2022.

\bibitem{buhmann2003radial-book}
M.D. Buhmann.
\newblock {\em Radial Basis Functions: Theory and Implementations}.
\newblock Cambridge University Press, 2003.

\bibitem{farrell-2017IMAJNA-multilevel}
P.~Farrell, K.~Gillow, and H.~Wendland.
\newblock Multilevel interpolation of divergence-free vector fields.
\newblock {\em IMA J. Numer. Anal.}, 37(1):332--353, 2017.

\bibitem{filbir-2008MathNachr-polynomial}
F.~Filbir and W.~Themistoclakis.
\newblock Polynomial approximation on the sphere using scattered data.
\newblock {\em Math. Nachr.}, 281(5):650--668, 2008.

\bibitem{floater-1996JCAM-multistep}
M.S. Floater and A.~Iske.
\newblock Multistep scattered data interpolation using compactly supported
  radial basis functions.
\newblock {\em J. Comput. Appl. Math.}, 73(1-2):65--78, 1996.

\bibitem{franz-wendland2023multilevel}
T.~Franz and H.~Wendland.
\newblock Multilevel quasi-interpolation.
\newblock {\em IMA J. Numer. Anal.}, 43(5):2934--2964, 2023.

\bibitem{ganesh2006quadrature}
M.~Ganesh and H.N. Mhaskar.
\newblock Quadrature-free quasi-interpolation on the sphere.
\newblock {\em Elec. Tran. Numer. Anal.}, 25:101--114, 2006.

\bibitem{gao-2020SISC-multivariate}
W.W. Gao, X.P. Sun, Z.M. Wu, and X.~Zhou.
\newblock Multivariate {M}onte {C}arlo approximation based on scattered data.
\newblock {\em SIAM J. Sci. Comput.}, 42(4):A2262--A2280, 2020.

\bibitem{gao-2024NM-quasi}
W.W. Gao, J.C. Wang, Z.J. Sun, and G.E. Fasshauer.
\newblock Quasi-interpolation for high-dimensional function approximation.
\newblock {\em Numer. Math.}, 156(5):1855--1885, 2024.

\bibitem{georgoulis-2013SISC-multilevel}
E.H. Georgoulis, J.~Levesley, and F.~Subhan.
\newblock Multilevel sparse kernel-based interpolation.
\newblock {\em SIAM J. Sci. Comput.}, 35(2):A815--A831, 2013.

\bibitem{golitschek-2001ConApprox-interpolation}
M.V. Golitschek and W.A. Light.
\newblock Interpolation by polynomials and radial basis functions on spheres.
\newblock {\em Constr. Approx.}, 17(1):1--18, 2001.

\bibitem{gomes2001approximation}
S.M. Gomes, A.K. Kushpel, and J.~Levesley.
\newblock Approximation in ${L}_2$ {S}obolev spaces on the 2-sphere by
  quasi-interpolation.
\newblock {\em J. Fourier. Anal. Appl.}, 7:283--295, 2001.

\bibitem{hesse-2017NM-radial}
K.~Hesse, I.H. Sloan, and R.S. Womersley.
\newblock Radial basis function approximation of noisy scattered data on the
  sphere.
\newblock {\em Numer. Math.}, 137(3):579--605, 2017.

\bibitem{hubbert-2023Adv-generalised}
S.~Hubbert and J.~J{\"a}ger.
\newblock Generalised {W}endland functions for the sphere.
\newblock {\em Adv. Comput. Math.}, 49(1):3, 2023.

\bibitem{ibanez2010construction}
M.J. Ib\'{a}\~{n}ez, A.~Lamnii, H.~Mraoui, and D.~Sbibih.
\newblock Construction of spherical spline quasi-interpolants based on
  blossoming.
\newblock {\em J. Comput. Appl. Math.}, 234(1):131--145, 2010.

\bibitem{iske-2005NumerAlgor-multilevel}
A.~Iske and J.~Levesley.
\newblock Multilevel scattered data approximation by adaptive domain
  decomposition.
\newblock {\em Numer. Algorithms}, 39(1):187--198, 2005.

\bibitem{jetter-1999MCoM-error}
K.~Jetter, J.~St{\"o}ckler, and J.D. Ward.
\newblock Error estimates for scattered data interpolation on spheres.
\newblock {\em Math. Comput.}, 68(226):733--747, 1999.

\bibitem{kempf-2023NM-high}
R.~Kempf and H.~Wendland.
\newblock High-dimensional approximation with kernel-based multilevel methods
  on sparse grids.
\newblock {\em Numer. Math.}, 154(3):485--519, 2023.

\bibitem{gia_2010SINUM_multiscale}
Q.T. Le~Gia, I.H. Sloan, and H.~Wendland.
\newblock Multiscale analysis in {S}obolev spaces on the sphere.
\newblock {\em SIAM J. Numer. Anal.}, 48(6):2065--2090, 2010.

\bibitem{legia_2012ACHA_multiscale}
Q.T. Le~Gia, I.H. Sloan, and H.~Wendland.
\newblock Multiscale approximation for functions in arbitrary {S}obolev spaces
  by scaled radial basis functions on the unit sphere.
\newblock {\em Appl. Comput. Harmon. Anal.}, 32(3):401--412, 2012.

\bibitem{lin_SINUM2021_distributed}
S.-B. Lin, Y.G. Wang, and D.-X. Zhou.
\newblock Distributed filtered hyperinterpolation for noisy data on the sphere.
\newblock {\em SIAM J. Numer. Anal.}, 59(2):634--659, 2021.

\bibitem{mcdiarmid-1989SurveyComb-method}
C.~McDiarmid.
\newblock On the method of bounded differences.
\newblock {\em Surveys in combinatorics}, 141(1):148--188, 1989.

\bibitem{montufar_2022FoCom_distributed}
G.~Mont{\'u}far and Y.G. Wang.
\newblock Distributed learning via filtered hyperinterpolation on manifolds.
\newblock {\em Found. Comput. Math.}, 22(4):1219--1271, 2022.

\bibitem{muller1966SphHarm}
C.~M\"{u}ller.
\newblock {\em Spherical Harmonics}.
\newblock Lecture Notes in Mathematics, ~Vol. 17, Springer-Verlag,~ Berlin,
  1966.

\bibitem{narcowich-1999ACHA-multilevel}
F.J. Narcowich, R.~Schaback, and J.D. Ward.
\newblock Multilevel interpolation and approximation.
\newblock {\em Appl. Comput. Harmon. Anal.}, 7(3):243--261, 1999.

\bibitem{narcowich-2007FoCM-direct}
F.J. Narcowich, X.P. Sun, J.D. Ward, and H.~Wendland.
\newblock Direct and inverse {S}obolev error estimates for scattered data
  interpolation via spherical basis functions.
\newblock {\em Found. Comput. Math.}, 7(3):369--390, 2007.

\bibitem{narcowich-2002SIMA-scattered}
F.J. Narcowich and J.D. Ward.
\newblock Scattered data interpolation on spheres: error estimates and locally
  supported basis functions.
\newblock {\em SIAM J. Math. Anal.}, 33(6):1393--1410, 2002.

\bibitem{ramming-2018MCoM-kernel}
T.~Ramming and H.~Wendland.
\newblock A kernel-based discretisation method for first order partial
  differential equations.
\newblock {\em Math. Comput.}, 87(312):1757--1781, 2018.

\bibitem{sharon-2023SISC-multiscale}
N.~Sharon, R.S. Cohen, and H.~Wendland.
\newblock On multiscale quasi-interpolation of scattered scalar-and
  manifold-valued functions.
\newblock {\em SIAM J. Sci. Comput.}, 45(5):A2458--A2482, 2023.

\bibitem{sloan2011polynomial}
I.~H. Sloan.
\newblock Polynomial approximation on spheres-generalizing de la
  {V}all{\'e}e-{P}oussin.
\newblock {\em Comput. Meth. Appl. Math.}, 11(4):540--552, 2011.

\bibitem{sloan_1995JAT_polynomial}
I.H. Sloan.
\newblock Polynomial interpolation and hyperinterpolation over general regions.
\newblock {\em J. Approx. Theory}, 83(2):238--254, 1995.

\bibitem{sloan_2012Geom_filtered}
I.H. Sloan and R.S. Womersley.
\newblock Filtered hyperinterpolation: a constructive polynomial approximation
  on the sphere.
\newblock {\em GEM Int. J. Geomath.}, 3(1):95--117, 2012.

\bibitem{sun-gao-sun}
Z.J. Sun, W.W. Gao, and X.P. Sun.
\newblock Spherical quasi-interpolation using scaled zonal kernels.
\newblock {\em IMA J. Numer. Anal.}, doi.org/10.1093/imanum/draf104, 2025.

\bibitem{sun-2022JSC-convergent}
Z.J. Sun, W.W. Gao, and R.~Yang.
\newblock A convergent iterated quasi-interpolation for periodic domain and its
  applications to surface pdes.
\newblock {\em J. Sci. Comput.}, 93(2):37, 2022.

\bibitem{sun2025githubcode}
Z.J. Sun, M.Y. Lv, and X.P. Sun.
\newblock {SKQI}.
\newblock \url{https://github.com/zhengjiesun/SKQI}, 2024.
\newblock Accessed: 2025-10-12.

\bibitem{usta-levesley2018multilevel}
F.~Usta and J.~Levesley.
\newblock Multilevel quasi-interpolation on a sparse grid with the {G}aussian.
\newblock {\em Numer. Algorithms}, 77:793--808, 2018.

\bibitem{wendland2004scattered}
H.~Wendland.
\newblock {\em Scattered data approximation}, volume~17.
\newblock Cambridge university press, 2004.

\bibitem{wendland-2010NM-multiscale}
H.~Wendland.
\newblock Multiscale analysis in {S}obolev spaces on bounded domains.
\newblock {\em Numer. Math.}, 116(3):493--517, 2010.

\bibitem{wu2005generalized}
Z.M. Wu and J.P. Liu.
\newblock Generalized {S}trang-{F}ix condition for scattered data
  quasi-interpolation.
\newblock {\em Adv. Comput. Math.}, 23:201--214, 2005.

\bibitem{wu1994shape}
Z.M. Wu and R.~Schaback.
\newblock Shape preserving properties and convergence of univariate
  multiquadric quasi-interpolation.
\newblock {\em Acta Math. Appl. Sin.}, 10:441--446, 1994.

\end{thebibliography}

\end{document}